\documentclass[10pt,letterpaper]{article}
\usepackage{graphicx}
\usepackage{amsfonts,amsbsy,amssymb,amsmath}
\usepackage{epsfig}
\usepackage{dsfont}
\usepackage{mathrsfs}
\usepackage{geometry}

\usepackage{float}
\usepackage{color}
\usepackage{mathtools}

\newcommand{\R}{\mathbb{R}}

\newcommand{\F}{\mathcal{F}}
\newcommand{\N}{\mathcal{N}}
\newcommand{\E}{\mathbb{E}}
\newcommand{\C}{\mathcal{C}}

\newcommand{\OH}{\mathcal{O}}
\newcommand{\PR}{\mathbb{P}}
\newcommand{\ST}{\mathcal{S}}

\DeclarePairedDelimiter\floor{\lfloor}{\rfloor}
\def\1{1\kern-.20em {\rm l}}
\usepackage{color}
\usepackage{natbib}
\usepackage{rotating}
\usepackage{epsfig}
\usepackage{ifpdf}
\usepackage[colorlinks=true,linkcolor=red,citecolor=blue]{hyperref}

\newtheorem{theorem}{Theorem}[section]

\newtheorem{corollary}[theorem]{Corollary}

\newtheorem{definition}[theorem]{Definition}

\newtheorem{lemma}[theorem]{Lemma}

\newtheorem{proposition}[theorem]{Proposition}
\newtheorem{remark}[theorem]{Remark}

\newenvironment{proof}[1][Proof]{\noindent\textbf{#1.} }{\ \rule{0.5em}{0.5em}}

\numberwithin{equation}{section}

\numberwithin{equation}{section}

\usepackage{xr-hyper}
\usepackage{hyperref}
\begin{document}


\title{\sf 
%
%
Generalized regression operator estimation for continuous time functional data processes with missing at random response
}

\author{{ \sc Mohamed Chaouch}$^{1}$ ~~ and~~ {\sc Na\^amane La\"ib}$^{2}$ \\
$^{1}$ Department of Mathematics, Statistics, and Physics.\\
 Qatar University, Qatar\\
e-mail: mchaouch@qu.edu.qa\\
$^{2}$ CY Cergy Paris Univerist\'e, Laboratoire AGM, UMR 8088 du CNRS\\
F-95000 Cergy, France. \\
e-mail:  naamane.laib@cyu.fr,  naamane.laib@sorbonne-univeriste.fr\\
}
\maketitle
\begin{abstract}
In this paper,  we are interested in  nonparametric
 kernel estimation of a generalized regression function, including conditional cumulative distribution and  conditional quantile  functions,  
based on an incomplete sample $(X_t, Y_t, \zeta_t)_{t\in \mathbb{ R}^+}$ copies of a {\it continuous-time  stationary ergodic process} $(X, Y, \zeta)$.
The predictor $X$ is valued in some {\it infinite-dimensional} space, whereas the real-valued  process  $Y$ is observed when $\zeta= 1$ and   missing whenever $\zeta = 0$. 
Pointwise and uniform consistency (with rates)
 of these estimators  as well as  a central limit theorem  are established.  
Conditional bias  and  asymptotic quadratic error are also provided. 
Asymptotic and bootstrap-based confidence intervals for the generalized regression function are also discussed.  A first simulation study is performed to compare the discrete-time to the continuous-time estimations. A second simulation is also conducted to discuss the selection of the optimal sampling mesh in the continuous-time case. Finally, it is worth noting that our results are stated under {\it ergodic} assumption without assuming any classical mixing conditions.

\bigskip 
\noindent {\bf Keywords}: \vspace{1mm} Asymptotic quadratic error, continuous time ergodic processes, confidence intervals, exchangeable bootstrap, functional data, generalized regression, missing at random.

\end{abstract}

\section{Introduction}\label{intro}

Let $({\cal E}, d)$ be  an {\it infinite-dimensional} space equipped with a semi-metric $d(\cdot, \cdot)$.  We consider some 
${\cal E}\times \mathbb{R}\times  \{0,  1\}$-valued stationary and ergodic continuous  time process
$\left\{ Z_t=(X_t, Y_t, \zeta_t), \ {t\in \mathbb{R}^+}\right\}$
defined on a probability space $(\Omega, {\cal A}, \mathbb{P})$
and observed at any time $t\in[0, T ]$, where $\zeta$ is an indicator process taking values zero or one at any instant $t.$ 
For any $y$ in a compact set $S\subset \mathbb{R}$,  let $\psi_y(\cdot)$  be a real valued Borel
function defined on $S\times \mathbb{R}$.  The generalized regression function of 
$\psi_y(Y)$ given $X=x$ is defined by
$m_\psi(x,y):=\mathbb{E}(\psi_y( Y_t)| X_t=x)=\mathbb{E}(\psi(y, Y_t)| X_t=x)$
which is supposed to exist for any $x\in {\cal E}$  and independent of $t$.

Here, the process $Z$ has the same distribution as $(X, Y, \zeta)$, where   $Y$ is a real integrable process and     $X$ is a continuous time process valued in the functional space ${\cal E}$. This 
 means that, for any fixed time $t=t_0$, $X_{t_0} \in {\cal E}$. 
 Therefore, if, for instance, we take ${\cal E}:=L^2[0,1]$  the space of square integrable functions defined on $[0, 1]$, then for any fixed $t_0$, the process  $X_{0}:=\{X_{t_0}(s): s\in [0,1]\}$   describes a curve. 
 In this paper, we assume that    the response variable $Y$ is subject to the {\it missing at random} mechanism (MAR). This means that for an available observed sample $(X_t, Y_t, \zeta_t)_{0\leq t \leq T}$,  $X_t$ is completely observed, whereas $\zeta_t =1$ if $Y_t$ is observed at time $t$ and $\zeta_t=0$ otherwise. 
 The random variables $\zeta$ and $Y$ are supposed to be conditionally independent given $X$, that is $\mathbb{P}(\zeta=1 | X=x, Y=y) = \mathbb{P}(\zeta=1 | X=x) := p(x)$ almost surely (a.s.).  The MAR phenomena of the response variable may occur in several situations. For instance, in survey sampling studies the  non-response is an increasingly common problem, where the missing response reaches rates of $25\%$  to $30\%$ or even higher (see, e.g., \cite{Sikov18}). In such case the missing data become a real source of bias in survey sampling estimation.  Another example where the response may be subject to the AMR phenomena is  the household electricity consumption monitoring. Indeed, the real time collection of intraday electricity consumption is now possible after the deployment of smart meters at the household level. The transmission of the information from the smart meter towards the information system goes usually through WIFI or optical fiber networks which are significantly dependent on the weather conditions. Therefore, a response variable such as the daily total electricity consumption might be subject to missing at random mechanism due to bad weather conditions. 


 The statistical  analysis involving missing data,
 based on an incomplete discrete sample $(X_i, Y_i, \zeta_i)_{1\leq i\leq n}$ where the data are assumed independent and identically distributed (i.i.d), 
has been considered by several authors,  see for instance \cite{Ling15}, 
\cite{Cheng94}, \cite{LR02}, \cite{N03}, \cite{T06}, \cite{Liang07}, \cite{E11}, \cite{CLL16}, and \cite{C09}. 
 However, less attention has been given to the case when the covariate is infinite dimensional and the response is missing at random. One can cite, for instance,  \cite{F13}, where the authors considered the estimation of the regression function based on an i.i.d.  random sample, and \cite{Ling15} extended  their work   to discrete time ergodic processes.

 
The literature on the estimation of the regression function based on a completely observed sample $(X_t, Y_t)_{t\geq 0}$ copies of a strongly mixing  {\it continuous time  stationary  process} $(X, Y) \in \mathbb{R}^d\times \mathbb{R}^{d^\prime}$ is very extensive.
 %
 One may refer to the monograph  by \cite{B98} and the references theirin. Some of these results are extended by 
 \cite{DL14} and  \cite{B17}  to  the case where the underlying real-valued process is stationary and ergodic.
 \cite{ChaouchLaib2019}  have obtained an explicit upper bound of the asymptotic mean square error (AMSE) of kernel regression estimator when the data are sampled from a real-valued continuous time process with a missing at random response. Some of these results are extended   to the case  where $X$  is valued in an infinite-dimensional space and  $Y$  is a  real-valued completely observed. Under the $\alpha$-mixing condition \cite{Maillot2008} established   the 
  convergence  with rates of the regression operator, whereas 
  \cite{CM14}  obtained a superoptimal mean square convergence  rate 
  of  
  the MSE of regression function for continuous functional  time process with irregular paths.

This paper aims to extend \cite{Maillot2008} and  \cite{CM14} work at several levels. First, we suppose that the continuous time process satisfies an ergodic assumption rather than an $\alpha$-mixing one. Therefore, the dependence condition we consider is more general and involves several processes which do not satisfy the mixing property. On the other hand, this paper extends results established in \cite{Ling15} in the context of discrete time functional data processes to the continuous time framework. Indeed, given $(X_t, Y_t, \zeta_t)_{t\in \mathbb{ R}^+}$ copies of a continuous-time  stationary ergodic process $(X, Y, \zeta)$, we estimate the general operator $m_\psi(x,y)$ which  including conditional distribution function and conditional quantiles. It is worth noting that such extension is not obvious since it requires an appropriate definition of $\sigma$-fields adapted to continuous time context. Such adaptation  is  crucial when using martingale difference tools to establish asymptotic properties of the estimator. Second, the response variable considered here is affected by the MAR mechanism and therefore is not completely observed as in \cite{Maillot2008}. Moreover, in contrast to \cite{Maillot2008} and  \cite{CM14}, we do not limit our study to the mean square convergence but we provide a more exhaustive inference on the regression operator estimator including pointwise and uniform almost sure convergence rate, identification of the limiting distribution of our estimator and provide two methods to build confidence intervals (central limit theorem and   the  bootstrap procedure). 
  
The rest of this paper  is organised as follows. In Section \ref{sec2} we present the framework adapted to  continuous time ergodic  processes  and introduce assumptions needed for establishing asymptotic results.  The main asymptotic properties of the estimator are discussed in Section \ref{sec3}. An illustration of the performance of the proposed estimator is discussed through simulated data in Section \ref{sec4}. Section \ref{sec5} discusses an application to conditional quantiles. Finally technical proofs are given in Section \ref{secProofs}.

 \section{Framework and  assumptions}\label{sec2}
To define the framework of our study, we need to introduce  some definitions. 
 Let $X=(X_t)_{t\in [0, \infty)}$ be a continuous time process defined on a measurable probability space $(\Omega, {\cal F}, \mathbb{P})$ and observed at any time $t\in[0, T]$.

\begin{definition}
Let $\tau$ be a nonnegative real number,  a measurable set $ A$ is $\tau$-invariant if $T^\tau(A)=A$ for any $\tau$-shift transformation $T^\tau$, i.e. $(T^\tau(x))_\lambda=x_{s+\lambda}$.
The process $X$ is said to be  $\tau$-ergodic if for any $\tau$-invariant set $ A$, $\mathbb{P}( A)=\mathbb{P}^2(A)$.  It has the ergodic property if there exists an invariant distribution $F(\cdot)$, such that for any measurable function $h(\cdot)$ and a random variable $\xi$ distributed  as $F(\cdot)$  satisfying  $\mathbb{E}(|h(\xi)|)<\infty$, we have
\begin{eqnarray}\label{ergdicit1}
 \lim_{T\to \infty}\frac{1}{T}\int_{-\infty}^{\infty}h(X_t)dt= \mathbb{E}(h(\xi)) \quad \mbox{almost surely} \ \mbox{(a.s.).}
 \end{eqnarray}
\end{definition}

From now on,  we will be working on the filtered probability space   $\left(\Omega, {\cal F}, \{ {\cal F}_t, t\geq 0\}, \mathbb{P}\right)$. For a positive real number $\delta$ such that $n=\frac{T}{\delta} \in \mathbb{N}$ and $j\in \mathbb{N}\cap[1,n]$, consider the $\delta$- partition $(T_j=j\delta)_{1\leq j \leq n}$ of the interval $[0, T]$. Furthermore, for $t>0$ and $1\leq j \leq n$, we define  the following  $\sigma$-fields:
\begin{eqnarray*}
	\F_{t-\delta} &:=& \sigma((X_s, Y_s, \zeta_s): 0 \leq s < t-\delta), \quad  \F_j = \sigma((X_s, Y_s, \zeta_s), 0\leq s < T_j),\\
	{\cal S}^\ell_{t,\delta}&: =& \sigma((X_s, Y_s, \zeta_s); (X_r, (1-\ell)Y_r): 0\leq s < t, \quad t\leq r \leq t+\delta, \; \ell\in \{0, 1\}).\\
\end{eqnarray*}

\vskip -8mm

Whenever $s < 0$, $\F_s$ stands for the trivial $\sigma$-field. Here,   $\zeta_s$ is a  standard Bernoulli process (see Section below). Notice that,  for any $\delta>0$ and $t>0$,  ${\cal S}^0_{t, \delta} \subset {\cal S}^1_{t, \delta}$, ${\cal F}_{t-\delta} \subset {\cal S}^\ell_{t-\delta, \delta} \subset {\cal S}^\ell_{t, \delta},  
$    
$\text{and for any} \  j\geq 2 \ \text{such  that} \  T_{j-1}\leq t\leq T_j,  \quad 
{\cal F}_{j-2} \subseteq {\cal F}_{t-\delta}  \subset {\cal S}^\ell_{t, \delta}$.
%
%
%

\noindent The estimator $\widehat{m}_{\psi,T}(x,y)$  adapted to MAR response  may be defined, for any $\zeta\neq 0$  and  $\int_0^T \zeta_t \Delta_t(x) dt \neq 0$,  by
\begin{eqnarray}\label{estimator}
\widehat{m}_{\psi,T}(x,y) = \frac{\int_0^T \zeta_t \psi_y(Y_t) \Delta_t(x) dt}{\int_0^T \zeta_t \Delta_t(x) dt},
\end{eqnarray}
where 
$\Delta_t(x)= K\left( \frac{d(x, X_t)}{h_T}\right)$,  $K(\cdot)$ is a kernel density function, $h_T$ is the smoothing parameter tending to zero as $T$ goes to infinity.  Let 
$Z_1(x):= \int_0^\delta \Delta_t(x) dt$ and define the  conditional bias as  
\begin{eqnarray}\label{Bais}
B_T(x,y) :=  \frac{\overline{m}_{\psi,T,2}(x,y)}{\overline{m}_{\psi,T,1}(x,y)} - m_\psi(x,y):= C_T(x,y) - m_\psi(x,y).
\end{eqnarray}
where, for $j=1,2$,
\begin{eqnarray}\label{mbarTi-1}
\overline{m}_{\psi,T,i}(x,y) := \frac{1}{n\E(Z_1(x))} \int_0^T \E\left\{\zeta_t (\psi_y(Y_t))^{i-1} \Delta_t(x) | \F_{t-\delta} \right\} dt.
%
\end{eqnarray}

\bigskip
 Let us now introduce the assumptions under which we establish our asymptotic results.

\begin{itemize}
\item[(A1)] ({\bf Assumptions on the kernel function}). 
 Let $K$ is a nonnegative bounded kernel of class $\mathcal{C}^1$ over its support $[0,1]$ such that $K(1) >0$. The derivative $K^\prime$ exists on $[0,1]$ and satisfies the condition $K^\prime(v) <0,$ for all $v\in [0,1]$ and $|\int_0^1 (K^j)^\prime(v) dv| < \infty$ for $j=1,2.$
\item[(A2)] ({\bf Assumptions related to the continuous time functional ergodic processes})

$(i_0)$ Let $\alpha_0$ be a   nonnegative real number  and 
 $x\in \cal{E}$. Suppose, for any $0\leq s < t \leq T$ such that $t-s \leq \alpha_0$,  there exists a nonnegative continuous random function $f_{t,s}(x)$ a.s. bounded by a deterministic function $b_{t,\alpha_0}(x)$. 

Moreover, let $g_{t,s,x}(\cdot)$ be a random function defined on $\R$, $f(x)$ is a deterministic nonnegative bounded function and $\phi(\cdot)$ a nonnegative real function tending to zero (as its argument tends to 0), and assume that. 
\begin{itemize}
\item[(i)] $F_x(u):=\PR\left(d(x,X_t) \leq u \right) = \phi(u)f(x) + o(\phi(u))$ as $u\rightarrow 0.$
\item[(ii)] For any $0\leq s\leq t$, 
 $F_x^{{\cal F}_s}(u):=  \PR^{{\cal F}_s} (d(x,X_t)\leq u)= \PR(d(x,X_t)\leq u | \F_s) = \phi(u) f_{t,s}(x) + g_{t,s,x}(u)$ with $g_{t,s,x}(u) = o_{a.s.}(\phi(u))$ as $u\rightarrow 0,$ $g_{t,s,x}(u)/\phi(u)$ a.s. bounded and

 $T^{-1}\int_0^T g_{t, t-\delta,x}(u) dt = o_{a.s.}(\phi(u))$ as $T \rightarrow \infty$ and $u\rightarrow 0.$
\item[(iii)] For any $x\in \cal{E}$: 
$\lim_{T\rightarrow\infty}\frac{1}{T}\int_0^T f_{t, t-\delta}(x) dt = f(x)$, a.s.
\item[(iv)] There exists a nondecreasing bounded function $\tau_0$ such that, uniformly in $u\in[0,1]$,
$$\frac{\phi(hu)}{\phi(h)}=\tau_0(u) +o(1)\quad \mbox{as} \  h \downarrow 0  \quad \mbox{ and} \ \  \int_0^1 (K(v))' \tau_0(v) dv < \infty.$$
\item[(v)] $T^{-1}\int_0^T b_{t,\alpha_0}(x) dt \rightarrow D_{\alpha_0}(x)$ as $T\rightarrow\infty$ with $0<D_{\alpha_0}(x)<\infty$. 
\end{itemize}
\item[(A3)] ({\bf Local smoothness and continuity conditions})

Suppose, for any $(y, t)\in  S\times  [0,T]$
and $r$ such that  and $t\leq r \leq t+\delta$: 
\begin{itemize}
\item[$(i)$] $\E\left(\psi_y(Y_{r})| \ST_{t,\delta}^1 \right) = \E\left(\psi_y(Y_{r})| X_{r} \right) = m(X_r,y)$ a.s.
\item[$(ii)$] $\exists \beta>0$ and a constant $c>0$ such that, for any $(x',x'')\in {\cal E}^2$,

 $|m(x',y) - m_y(x'',y)| \leq c d^\beta(x', x'').$
\item[$(iii)$] For any $ \kappa_1\geq 2$, $\E\left( |\psi_y(Y_r)|^{\kappa_1} | \ST_{t,\delta}^1\right) = \E\left(|\psi_y(Y_r)|^{\kappa_1} | X_r \right)$ a.s. 

The functions  $W_{\kappa_1}(x,y) := \E\left(|\psi_y(Y)|^{\kappa_1} | X=x \right)$ and  $\overline{W}_{\kappa_1}(x,y) := \E\left(|\psi_y(Y) - m_\psi(x,y)|^{\kappa_1} | X=x \right)$ are continuous in the neighbourhood of $x$ and $\sup_{x\in {\cal C}, y\in S}|W_{\kappa_1}(x,y)|<\infty$ a.s.

\item[$(iv)$] $\E(\zeta_t | {\cal S}_{t,\delta}^0) = \E(\zeta_t | X_t)= p(X_t)$ a.s.  

For any $x\in \cal{E}$, $\sup_{\{x': d(x,x')\leq u\}} |p(x') - p(x)| = o(1)$ a.s. as $u\rightarrow 0.$ 

\item[$(iv^\prime)$]  For any $x\in {\cal E}$, $ \kappa_1\geq 2$, $\E(|\zeta_t|^{\kappa_1} | X_t=x): = U_\kappa(x)$, and 

 $\sup_{\{x': d(x,x')\leq u\}} |U_\kappa(x') - U_\kappa(x)| = o(1)$ a.s. as $u\rightarrow 0.$
\end{itemize}
\end{itemize}
For $j=1, 2$, define the following moments, which    are independent  of
$x \in {\cal E}$
\begin{eqnarray}\label{mbarTi}
M_j=K^{(j)}(1)-\int_0^1 \left( K^j \right)^\prime (u) \tau_0(u)du.
\end{eqnarray}

\subsection{ Comments on the  assumptions}

Condition (A1) is related to the choice of the kernel $K$, which  is very usual in nonparametric functional estimation. 
Condition (A2)(i)-(ii) reflects the ergodicity property assumed on the continuous time functional process. It plays an important role in studying the asymptotic properties of the estimator. The functions   $f_{t, s}$ and $f$  play the same role as the conditional and unconditional densities in finite dimensional case, whereas $\phi(u)$
characterizes the impact of the radius $u$ on the small ball probability  as  $u$ goes to $0$.  
Several examples to satisfy  these conditions are given in \cite{LL10} for discrete time functional data process. 

Some examples are also given to satisfy this condition in  \cite{DL14}  for the case where the observations $(X_t, Y_t)$ are sampled from  an {\it ergodic continuous times} process  taking values in $\mathbb{R}^d \times \mathbb{R}$ space. 

 Condition (A2)-(iii) involves the ergodic nature of the process where the random function $f_{t, t-\delta}$ belongs to the space of continuous functions ${\cal C}^0$. 
   Approaching the integral  $\int_0^{T} f_{t, t-\delta}(x) dt$ by its Riemann's sum:  $ T^{-1}\int_0^{T} f_{t, t-\delta}(x) dt \simeq n^{-1}\sum_{j=1}^{n} f_{j\delta, (j-1)\delta}(x)$, it is easy to prove that the sequence $(f_{j\delta, (j-1)\delta}(x))_{j\geq 1}$  is stationary and ergodic (see, \cite{DL14}). (A2)-(iv) is an usual condition when dealing with functional data. Assumption (A2)-(v) is a consequence of ergodic assumption. 
Condition (A3) is a Markov-type condition and characterizes the conditional moments of $\psi_y(Y)$. 
 It is satisfied when considering, for instance,   the model $\psi_y(Y_t)=m_\psi(X_{t})+\epsilon_t$ where $(\epsilon_t)$ is a square integrable process independent  of $(X_t)$ for  any $t\geq 0$.

 (A3)(iv) 
 assumes the continuity of the conditional probability of observing a missing response.  
 The moments $M_j$ are linked to  the small probability function $\tau_0$. One can refer to \cite{F07} for a  discussion on the choice of $\tau_0$, the Kernel $K$ and the positivity of $M_j$. 



\section{Main results}\label{sec3}
In this section we investigate several asymptotic properties of the continuous time generalized regression estimator. Some particular cases, related to specific choices of the function $\psi_y(\cdot)$, including the conditional cumulative distribution function and the conditional quantiles will also be discussed. 
\subsection{Almost sure consistency rates }\label{subsec2-1}
\subsubsection{Pointwise consistency}
%
The following theorem establishes an almost sure pointwise consistency rate of $\widehat{m}_{\psi,T}(x,y).$
\begin{theorem}\label{thm1}{\rm ({\bf Pointwise consistency})}
		Assume that (A1)-(A3) hold true and the following conditions are satisfied 	
	\begin{eqnarray}\label{cond}
	\lim_{T\rightarrow\infty}T\phi(h_T) =\infty \quad\quad \mbox{and} \quad\quad \lim_{T\rightarrow\infty}\frac{\log T}{T\phi(h_T)} = 0.
	\end{eqnarray}
	Then, we have for $T$ sufficiently large that   	
	\begin{eqnarray}\label{Th-PointwiseConv}
	\widehat{m}_{\psi,T}(x,y) - m_\psi(x,y) = \OH(h_T^\beta) + \OH\left(\sqrt{\frac{\log T}{T\phi(h_T)}} \right).
\end{eqnarray}
\end{theorem}

\begin{remark}
	{\rm 
(i) Theorem \ref{thm1} generalizes Theorem 1 of \cite{LL2011}
 established in the context of discrete time stationary ergodic process and Theorem 3.4 of  \cite{F-al05} stated under a mixing assumption
 with completely observed response and the support of $y$ is reduced to one point. 

 (iv) The function $\phi(h_T)$ can decrease to zero at an exponential rate, whenever   $h_T$  goes  to zero,  therefore  $h_T$ should be  chosen  decreases to zero   at a logarithmic rate.
}
\end{remark}

\subsubsection{Uniform consistency}
To establish the uniform consistency with rate of the regression operator, we need   additional   assumptions  that allow to express the  uniform convergence rate as a function of the entropy number. 
Let ${\cal C}$ and  $S$ be compact sets in ${\cal E}$ and $\mathbb{R}$, respectively.
Consider, for any $\epsilon >0$,
let the $\epsilon$-covering number of the compact set ${\cal C}$, say, $N_\epsilon:=\N(\epsilon, \C, d)$, which measures how full is the class $\C$, and defined as:
\begin{eqnarray*}
N_\epsilon&:=& \min \left\{n: \mbox{there exist} \; c_1, \dots, c_n \in \C\; \mbox{such that}\; \forall x\in \C \;\mbox{we can find } \right.\\
	&&\left. \quad\quad\quad\quad 1\leq i  \leq n \;\mbox{such that}\; d(x, c_i) < \epsilon \right\}.
\end{eqnarray*}

The finite set of points $c_1,c_2, \ldots,c_{N_\epsilon}$ is called
an $\epsilon$ -net of ${\cal C}$  if ${\cal C} \subset  \cup_{k=1}^{N_\epsilon} B(c_k, \epsilon)$, where  $B(c_k, \epsilon)$  is the ball, with respect to the topology induced by the semi-metric $d(\cdot,\cdot)$, centred at $c_k$ 
and of radius $\epsilon$.  The quantity 
$\varphi_{\cal C}(\epsilon)=\log(N_\epsilon)$ is called the Kolmogorov's $\epsilon$-entropy of the set  ${\cal C}$ that may be seen as a  tool to measure the complexity of the subset ${\cal C}$,  in the sense that high entropy means that a large amount of information is needed to describe  an element of ${\cal C}$ with an accuracy $\epsilon$.
Several examples of  $\varphi_{\cal C}(\epsilon)$ 
covering special cases of functional process   are  given in \cite{F10}
and \cite{LL2011}.
\begin{itemize}
	\item[(U0)]  Assume that  (A2) holds uniformly, in the following sense: 	
	\begin{description}	
		\item[(i)] (A2)(i) and (A2)(ii)  hold true  with the remaining term $o(\phi(u))$    is uniform  in $x$,
		\item[(ii)] 
		For any $x\in {\cal C}$, 
		$\lim_{T \rightarrow \infty}\sup_{x \in {\cal C}}\big| \frac1T \int_0^T f_{t,t-\delta}(x)dt-f(x)\big| =0 \ a.s.$
		\item[(iii)] $T^{-1}\int_0^T b_{t,\alpha_0}(x) dt \rightarrow D_{\alpha_0}(x)$ as $T\rightarrow\infty$ with $0<
		\sup_{x\in {\cal C}}D_{\alpha_0}(x)<\infty$. 
		\item[(iv))] 	
		$b_0<\inf_{x\in \C} f(x) \leq  \sup_{x\in \C} f(x) \ <\infty$ for some  nonegative real number $b_0$.
		\item[(v)] $\inf_{x\in C} p(x) > b_1$ for some  nonegative real number $b_1$.
	\end{description}
\end{itemize}
\begin{itemize}
	\item[(U1)] The kernel function $K$ satisfies the following conditions:
	\begin{itemize}
		\item[(i)] K is a H\"older function of order one  with a constant $a_K$,
		\item[(ii)] There exist two constants $a_2$ and $a_3$ such that $0<a_2 \leq K(x) \leq a_3 <\infty$, for any $x\in \cal{C}.$
	\end{itemize}
	\item[(U2)] For $1\leq \ell \leq 2$, the sequence of random variables $(\psi^\ell_y(Y_t))_{t}$ is ergodic and $\E\left(|\psi^\ell_y(Y_0)| \right) < \infty.$
	
	%
	%
	
	\item[(U3)]	There exist $c_\psi> 0$ and nonnegative real number  $\gamma$   such that 
	for any $y\in S$ 
	$$\sup_{y^\prime \in[y-u,y+u]\cap S}\vert \psi_y(Y_t)-\psi_{y^\prime}(Y_t)\vert
	\leq  c_\psi  u^{\gamma}.$$
	%
	%
	\item[(U4)]	 Let $T_n$ be the integer pat of $T$ and suppose   for  $T$ large enough that
	$$\frac{(\log T)^2}{T\phi(h_T)}	<\varphi_C(\epsilon_n)<\frac{T\phi(h_T)}{\log T} \qquad \text{with} \quad  \epsilon_n=\frac{ \log T_n}{T_n}.$$
\end{itemize}

Conditions (U0)  are  standard  in this context in order to get uniform consistency rate. Condition (U1) is usually used   when we deal with nonparametric estimation  for functional data, whereas (U2) requires 
to the existence of up to the order two moments.   Hypothesis  (U3) is a regularity condition upon the function $\psi_y(\cdot)$ which is necessary to
obtain a uniform result over the compact $S$.
Hypothesis (U4)  allows to  cover the subset  ${\cal C}$  with a finite number of balls and to  express the convergence rate in terms of the Kolmogorov's entropy of this subset. Similar condition  has been used in  \cite{F10}, where the authors have pointed out that
for a radius not too large, one requires the quantity $\varphi_C(\log T_n/T_n) $
is not too small and not too large. This condition  seems to satisfy this exigence, since  it  implies that   $\frac{\varphi_C(\log T_n/T_n)}{T\phi(h_T)} $ goes to $0$ for sufficiently large  $T$,  in addition   some  examples given in \cite{F10}
and \cite{LL2011} satisfy  (U4). 

\vskip 2mm
Theorem \ref{uniform} states uniform  consistency rate of the kernel regression estimator. It generalizes  
Theorems  2 in \cite{LL2011} established in the context of  discrete time stationary ergodic process with  completely observed response.	
\begin{theorem}\label{uniform} {\rm ({\bf Uniform consistency})}.
	Assume  (A1), (U0)-(U4), (A3) hold true.
	Moreover, suppose  Conditions (\ref{cond}) are satisfied and 
	%
	\begin{eqnarray}\label{CEntropy}
	\sum_{n\geq 1} n^{\gamma} \exp\{(1-\eta)\varphi_{\cal C}(\frac{\log n}{n})\}<\infty \quad \text{for some} \  \eta>0 \ \text{where} \ \gamma 
	\ \text{is as in } \ (U3).
	\end{eqnarray}	 
	Then, we have 
	\begin{eqnarray}\label{V-unifome1}
	\sup_{y\in S}\sup_{x\in \C} |\widehat{m}_{\psi,T}(x,y) - m_\psi(x,y)| = \OH_{a.s.}\left(h_T^\beta \right) + \OH_{a.s.}\left(\sqrt{\frac{\varphi_{\cal C} (\epsilon_T)}{T\phi(h_T)}} \right) \quad 
	\text{as} \quad T\to \infty.
	\end{eqnarray}
\end{theorem}

\subsection{Asymptotic conditional bias  and risk evaluation}\label{subsec2-2}

Before evaluating the conditional bias, let us introduce some additional notations. Consider for $i=1,2$, the following assumptions:
\vskip 1mm

\noindent ({\bf BC1}) Let $d_t(x)=d(X_t, x)$ and suppose, for any  $t\geq 0$, that
$$\mathbb{E}\left[ m_\psi(X_t, y)-m_\psi(x,y) \ | d_t(x), \ {\cal F}_{t-\delta}\right]=\mathbb{E}\left[ m_\psi(X_t,y)-m_\psi(x,y) \ | 
d_t(x)\right]:=\Psi_y(d_t(x)),$$
 the function $\Psi_y$ is differentiable at $0$ and satisfies  $\Psi_y(0)=0$ and $\Psi_y^\prime(0) \neq 0$ for any $y\in \mathbb{R}$. 
%
This condition was introduced  in \cite{F07} and  used by \cite{LL10} to evaluate the conditional bais.  The introduction of $\psi_y(\cdot)$ allows to make an  integration with respect to the real  random variable $d_t(x)$  rather than with respect to the couple of random variables $(d_t(x), X_t)$,  where $X_t$ being functional continuous random variable.

The following Proposition gives  an asymptotic  expression of the conditional bias term, which  generalizes Proposition 1 of \cite{LL10} for discrete time estimator
to our setting. Its  proof is similar and therefore is omitted. 
\begin{proposition}[{\bf Conditional Bias}]\label{bias}
	Under assumptions (A1)-(A3), (BC1) and conditions  (\ref{cond}),  we have
	\begin{eqnarray*}
		B_T(x,y) &=& \frac{h_T \Psi_y^\prime(0)}{M_1} \left[K(1) - \int_0^1 (sK(s))^\prime \tau_0(s) ds + o_{a.s.}(1) \right] + \OH_{a.s.}\left( h_T\sqrt{\frac{\log T}{T\phi(h_T)}} \right). 
	\end{eqnarray*}
\end{proposition}
The next result gives an explicit expression of the asymptotic quadratic risk of the estimator $\widehat{m}_{\psi,T}(x,y)$.
\begin{theorem} [{\bf Quadratic risk }]\label{MSE}
	Suppose that assumptions (A1)-(A3) and condition (\ref{ergdicit1})  hold true. Then, we have, for a fixed $(x, y)\in {\cal E}\times \mathbb{R} $, whenever $p(x)>0$ and
	$f(x)>0$,  that
	%
		\begin{eqnarray*}
\mbox{\textsc{MSE}}(x,y):=\E\left[ \left( \widehat{m}_{\psi,T}(x,y) - m_\psi(x,y) \right)^2 \right]
=	A_1 h^2_T\left[ A_1+    \OH_{a.s.}\left(\sqrt{\frac{\log T}{T\phi(h_T)}} \right) \right]
+\frac{A_2(x,y)}{T\phi(h_T)},
	\end{eqnarray*}
%
$$A_1=\frac{\Psi_y^\prime(0) }{M_1} \left[K(1) - \int_0^1 (sK(s))^\prime \tau_0(s) ds +o(1) \right] \quad \text{and} \ 
A_2(x,y)=\frac{4 \left( W_2(x,y) +(m_\psi(x,y))^2 \right) M_2}{p(x) M_1^2 f(x)}.$$
\end{theorem}

\begin{remark}. 
{\rm 
(i)	Notice that  for sufficiently large  $T$, the expression of MSE becomes $A_1^2 h^2_T+ \frac{A_2(x,y)}{T\phi(h_T)}$.  This result generalizes the one in \cite{ChaouchLaib2019}  established  in the framework of  real-valued continuous time processes where the bias term was of order $h_T^2$. Notice however that the  bias  term  obtained in \cite{ChaouchLaib2019}  is of  order $h_T^2$ which is smaller  than  $h_T$ given in Proposition \ref{bias}.  This increase in the bias term is because of the infinite dimensional characteristic of the functional space. 

(ii)   The mean squared error  can be used as a
theoretical  guidance to select the ``optimal" bandwidth
by  minimizing the quantity  $A_1^2 h^2_T+ \frac{A_2(x,y)}{T\phi(h_T)}$ with respect to $h_T$. The terms $A_1$ and $A_2(x,y)$ have unknown  explicit from; and should be replaced by their empirical consistent estimators:
 $\Psi_{y, T}^\prime(0), \ (M_{j,T})_{j=1,2}, \ \tau_{0, T}, \ p_T, W_{2,T}, \ f_T.$
  Notice that, $\Psi_y^\prime(0)$ may be viewed  as real regression function with response  variable $m_\psi(X,y)-m_\psi(x,y)$ and  covariable  $d(X,x)$, it may be then estimated by a  kernel regression estimate
   $\Psi^\prime_{y, T}(0)$ by replacing 
   $m_\psi(x,y)$  by  its estimator $\widehat{m}_{\psi,T}(X,y)$.
}
\end{remark}

\subsection{Asymptotic normality }\label{subsec2-3}

\begin{theorem}\label{AsymNormality} 
	Assume  conditions (A1)-(A3) are fulfilled and 
	\begin{eqnarray}\label{ConTheoremNor}
	\lim_{T \rightarrow \infty}T\phi(h_T)=+\infty, \ 
	h_T^\beta\sqrt{T\phi(h_T)}=o(1) \quad \text{and} \quad 
	h_T^\beta \log T^{1/2}=
	o(1) \quad \text{as} \ T \rightarrow \infty.
	\end{eqnarray}
	
	Then,  we have,  for any $(x, y)\in {\cal E}\times S$ such that $f(x)>0$
	
	$$\sqrt{T\phi(h_T)}(\widehat{m}_{\psi,T}(x,y)-m_\psi(x,y)) \quad \xrightarrow{d}\quad {\cal N}(0, \sigma^2(x, y)),$$
	where 
	\begin{eqnarray}\label{sigma1}
\sigma^2(x, y) \leq \frac1{f(x)}\frac{M_2}{M_1^2  p(x)}\overline{W}_2(x, y):=\frac1{f(x)}\widetilde{V}(x,y)  \quad 
	\text{as} \quad T\longrightarrow \infty,
		\end{eqnarray}
		and
	$$\overline{W}_2(x,y)= \mathbb{E}\left[ \left(\psi_y(Y)-m_\psi(x,y)\right)^2 \ | X=x  \right]. 
	$$		
\end{theorem}

Note that, the statement (\ref{sigma1}) gives only an upper bound of  the asymptotic variance  $\sigma^2(x, y)$.  The following proposition  gives an estimate of $\widetilde{V}(x,y)$  that will be needed to construct confidence bounds for the unknown operator $m(x, y)$. 
\begin{proposition}\label{PropoV(x)} Suppose 
 conditions of  Theorem \ref{AsymNormality}  hold and $\sigma^2(x,y)>0$, then   
	$$\widetilde{V}_T(x,y):=\frac{\sqrt{M_{T,2}}}{M_{1, T}}\sqrt{\frac{\overline{W}_{2,T}(x,y)}
		{T F_{x,T}(h_T) p_T(x)}},$$
	is an asymptotically  consistent estimator for $\widetilde{V}(x,y)$. The quantities  $M_{1,T}$, $M_{2,T}$, $\overline{W}_{2,T}$, $p_T(x)$ and  $F_{x,T}$ are empirical versions of   $M_{2}$, $M_{1}$, 
	$\overline{W}_{2}$, $p(x)$ and  $F_{x}$ respectively. 
\end{proposition} 
The estimators  $M_{1,T}$ and $M_{2,T}$ depend on the 
unknown  quantity $\tau_0$ given in (A2)(iv) that may be  estimated by:
$$\tau_{0,T}(u)=\frac{F_{x,T}(uh_T)}{F_{x,T}(h_T)} \quad \text{with}  \quad F_{x,T}(u)=\frac1T\int_0^T 
\1_{\left\{ d(x,X_t)\leq u \right\}} dt,$$
whereas $\overline{W}_{2,T}$  can be deduced from a  nonparametric estimator for the  conditional variance function $\sigma^2_T$ 
(see \cite{LL10}).

\subsection{Continuous time confidence intervals}\label{subsec2-4}
\subsubsection{Asymptotic confidence intervals}

 Using the non-decreasing property of the cumulative  standard Gaussian distribution function and  the estimator $\widetilde{V}_T(x, y)$  we can then,   with the help of  Proposition \ref{PropoV(x)} and Theorem \ref{AsymNormality}, 
	constructed  confidence bands for the regression function $m_\psi(x,y)$, which  are similar  to the those  obtained in the discrete time case. This is the subject of the following corollary. 

\begin{corollary}\label{colorlarryNor}
	Assume conditions of Theorem (\ref{AsymNormality}) are fulfilled and conditions in (\ref{ConTheoremNor}) are replaced by
		\begin{eqnarray}\label{Corollay-intr-1}
	\lim_{T\rightarrow \infty} TF_{x,T}(h_T)=+\infty \quad \text{and} \quad \lim_{T\rightarrow \infty} h_T^\beta\sqrt{TF_{x,T}(h_T)}=0.
\end{eqnarray}	
	 Then, for any $0<\alpha<1$, the $(1-\alpha)100\%$ confidence bands for  $m_\psi(x,y)$ are given by 
	\begin{eqnarray}\label{Corollay-intr}
		\widehat{m}_{\psi,T}(x,y) \pm c_{1-\alpha/2} \frac{M_{T,1}}{\sqrt{M_{T,2}}}\sqrt{\frac{\overline{W}_{T,2}(x,y)}{T F_{x,T}(h_T) p_T(x)}}, \quad \text{as} \ T\to \infty,
	\end{eqnarray}
	where $c_{1-\alpha/2}$ is the  quantile of standard  Gaussian distribution.
\end{corollary}

\subsubsection{Exchangeable bootstrap-based confidence intervals}
In a variety of statistical problems, the bootstrap provides a simple method  for circumventing technical difficulties due to the intractable distribution theory (this is the case in
Theorem \ref{AsymNormality}, where the limiting law depends crucially on the unknown variance). Resampling techniques become a
powerful tool, e.g., for setting confidence bands, that we will illustrate in the sequel. The key idea behind the bootstrap is that if a sample is representative of the underlying population, then one can make inferences about the population characteristics by resampling from the current sample. To apply this approach for constructing  confidence bands we have to define a bootstrap version  that has the same asymptotic distribution  as  $\sqrt{T\phi(h_T)}(\widehat{m}_{\psi,T}(x,y)-m_\psi(x, y))$.   
Inspection of the proof of Theorem \ref{AsymNormality} shows that  the principal term   leading to establish the asymptotic normality
is given by $\frac{\sqrt{T\phi(h_T)}Q_T(x,y)}{\widehat{m}_{\psi,T,1}(x,y)}$, where $Q_T$ is defined in  (\ref{Q}).  Therefore, a bootstrap version  for  this quantity will be sufficient for estimating   the true quantile $c_{1-\alpha/2}$ for a given $0<\alpha<1$. 
This may be done by combining Theorem \ref{AsymNormality}, 
 Theorem 2.1. of \cite{Pauly2011} for martingale difference arrays sequences and Theorem 3 of  \cite{B16}. In this respect, define 
\begin{eqnarray*}
S_T(x,y):=\sqrt{T\phi(h_T)}Q_T(x,y)&=&\sum_{i=1}^n \xi_{T,i}(x,y), 
\end{eqnarray*}
where
$$ 
\xi_{T,i}(x,y)=\eta_{T,i}(x,y)-\E\left[\eta_{T,i}(x,y) \ \Big| {\cal F}_{t-\delta} \right],$$
and
\begin{eqnarray}\label{NormalityProofLema1}
\eta_{T,i}(x,y) &=& \frac1{ \E Z_1}\sqrt{\frac{\phi_T(h)}{n}}
\int_{T_{i-1}}^{T_i}\zeta_t\Delta_t(x) \left[\psi_y(Y_t)-m_\psi(x,y)\right] dt.
\end{eqnarray}
\noindent For every fixed $n\geq 1$, the sequence $(\xi_{T,i})_{1\leq i \leq n}$ is a centred martingale difference array  with respect ${\cal F}_{i-1}$(see the proof of Lemma \ref{AsymNormalityQ}).

Define the  resampling version of the quantity $S_T(x,y)$ as 
\begin{eqnarray}\label{BosstrabConVer}
S^\star_T(x,y)&=&
\sqrt{n}\sum_{i=1}^n  W_{T, i} \xi_{T,i}(x,y),
\end{eqnarray}
where 
  $W_{T, i}$ are weight functions  satisfying  conditions $W.1-W.4$  in  \cite{B16}. 
The following proposition, which is a consequence of  Theorem \ref{AsymNormality}, Lemma \ref{mhat} and the Slutsky's lemma,  states 
the asymptotic normality of the statistic $S^\star_T$. Its proof may be obtained similarly to that of  Theorem 3 of  \cite{B16}. 
\begin{proposition}\label{poropBoostrab} Assume that the above conditions related  to  bootstrap weights hold
and  the sequences of random variables
$\left\{ W_{T, i}: \ 1\leq i \leq n \right\}$ and  
$\left\{ \xi_{T, i}: \ 1\leq i \leq n \right\}$ are independent. 
	Then under the assumptions of Theorem \ref{AsymNormality} and for any $x \in \cal E$  such that $f(x) > 0$, we have 
	\begin{eqnarray}\label{propositionBost}
		S^\star_T(x,y) \quad \xrightarrow{d}\quad {\cal N}(0, \sigma^2(x,y)) \  \text{as} \quad T\to \infty, \  
	where \  \sigma^2(x,y)\  \text{ is as in Theorem}.  \  \ref{AsymNormality}	
	\end{eqnarray}	
\end{proposition}
\begin{remark}\label{remarkBoostrab}. 
{\rm Observe that $S^\star_T$ depends on the  unknown function $m_\psi(x,y)$.
Substitute  $m_\psi(x,y)$ by  $\widehat{m}_{\psi,T}(x,y)$ and  $S^\star_T$ by its  estimate version  $\widehat{S^\star_T}$, it follows from   Lemma  \ref{lem1} and Theorem \ref{thm1} that
$\widehat{S^\star_T}-S^\star_T=o(1)$ as $T$ goes to infinity. 
Therefore the  statement (\ref{propositionBost}) still holds if we replace $S^\star_T$ by  $\widehat{S^\star_T}$. }
\end{remark}
Proposition \ref{poropBoostrab} turns out to be useful for  obtaining confidence bands for the regression function  $m_\psi(x, y)$. 
 Indeed, let  $0 < \alpha < 1$  be given, then, under conditions of Theorem \ref{AsymNormality}, we have 
\begin{eqnarray}\label{boostrabInterval}
	&&\lim_{T \rightarrow \infty}\mathbb{ P}\left(\widehat{m}_{\psi,T}(x,y)-c_{1-\alpha/2}(T \phi(h_T))^{-1/2}\leq m_\psi(x,y)\leq \widehat{m}_{\psi,T}(x,y)+c_{1-\alpha/2}(T \phi(h_T))^{-1/2}\right)
	\nonumber
	\\&&~~~~~~= \mathbb{ P}\left(| {\cal N}(0,{\sigma}^2(x,y) )|\leq c_{1-\alpha/2}\right),
\end{eqnarray}
where $c_{1-\alpha/2}$ is the $\alpha$-quantile of ${\cal N}(0,\sigma^2(x, y))$. In practice the true quantile $c_{1-\alpha/2}$ cannot be computed since the variance $\sigma^2(x, y)$ is unknown.  Nevertheless, from
Proposition \ref{poropBoostrab}, we can approximate the true quantile $c_{1-\alpha/2}$
 by means of the quantiles of the bootstrapped distribution of $\widehat{S^\star_T}$ (see Subsection \ref{subsec2-5}  for details).

\subsection{Sampling schemes and computation of the  confidence intervals}\label{subsec2-5}

In the previous section, the process was supposed to be observable  over $[0, T]$. However, in  practice the data are  often collected  according to a sampling scheme since it is difficult to observe a path continuously at any time $t$ over the interval $[0, T]$. 

Hereafter, we briefly discuss the effect of a sampling scheme on the construction of  confidence intervals for the regression function $m_\psi(\cdot, \cdot)$. Data sampling from a continuous time process at instants $(t_k)_{k=1, \dots, n}$ can be made regularly, irregularly or even randomly. For a sake of simplicity, we consider here the case where the instants $(t_k)$ are irregularly spaced,  that is $\inf_{1\leq k\leq n}|t_{k+1}- t_k|=\delta >0.$
Now, for  $\ell\in \{0, 1\}$ and $k\in \{1, \ldots n\}$, we define the following increasing families of $\sigma$-algebra:
\begin{eqnarray*}
\F_{k}:=	\F_{t_k} = \sigma\left( (X_{t_1}, Y_{t_1}, \zeta_{t_1}), \dots, (X_{t_k}, Y_{t_k}, \zeta_{t_k})\right),
\end{eqnarray*}
and
\begin{eqnarray*}
{\cal G}^\ell_{k}:=	{\cal G}^\ell_{t_k}= \sigma\left(\left(X_{t_1}, Y_{t_1}, \zeta_{t_1}\right), \dots, \left(X_{t_k}, Y_{t_k}, \zeta_{t_k}\right); X_{t_{k+1}}, (1-\ell)\zeta_{t_{k+1}}\right).
\end{eqnarray*}
 The purpose then consists in estimating $m_\psi(\cdot,\cdot)$ given the discrete time ergodic stationary process $\left(X_{t_k}, Y_{t_k},\zeta_{t_k} \right)_{k=1, \dots, n}$ sampled from the underlying continuous time process $\left\{(X_t, Y_t, \zeta_t): 0\leq t \leq T \right\}$. In case of a regular sampling scheme, that is $T=n\delta$, the  estimator $\widehat{m}_{\psi,T}(x,y)$ defined in (\ref{estimator}) becomes
\begin{eqnarray}\label{estimator-sampl}
\widehat{m}_{\psi,n}(x,y) = \frac{\sum_{k=1}^n \zeta_{t_k} \psi_y(Y_{t_k}) K\left( \frac{d(x,X_{t_k})}{h_n}\right) }{\sum_{k=1}^n \zeta_{t_k}  K\left( \frac{d(x,X_{t_k})}{h_n}\right)}, \quad\quad t_k = k\delta.
\end{eqnarray}
%
%

%
%


\bigskip 
\noindent a)  {\bf Asymptotic confidence intervals}

Notice that Theorem \ref{ConTheoremNor} sill holds for the estimate  $\widehat{m}_{\psi,n}(x,y)$ when replacing $T$ by $n \delta$.  The limiting law is a Gaussian  random variable with mean zero and variance function $\sigma^2(x, y)= \frac1{f(x)}\frac{M_2}{M_1^2  p(x)}\overline{W}_2(x, y).$  
Making use of  Corollary \ref{colorlarryNor} and considering similar steps as in \cite{LL10}, it follows that, for any $0<\alpha<1$, the asymptotic confidence intervals of $m_\psi(x,y)$ is given as:  
\begin{eqnarray*}
	\widehat{m}_{\psi,n}(x,y) \pm c_{1-\alpha/2} \frac{M_{n,1}}{\sqrt{M_{n,2}}}\sqrt{\frac{\overline{W}_{n,2}(x,y)}{n F_{x,n}(x) p_{n}(x)}}, \quad \text{as} \ n\to \infty,
\end{eqnarray*}
where $c_{1-\alpha/2}$ is the  quantile of standard  Gaussian distribution. 

\bigskip
\noindent b) {\bf Exchangeable bootstrap-based confidence intervals}\label{WBE}

The discrete-time   resampling version of the quantity $\widehat{S^\star_T}(x,y)$  is defined  for any  fixed $x\in \cal{E}$ as 
$$
\widehat{S^\star_n}(x,y)=\sqrt{n}\left\{\sum_{k=1}^{n}\left(W_{n,t_k}-\overline{W}_{n}\right)  \widehat{\xi}_{n, t_k}^\star(x,y) \right\},
$$
where $\overline{W}_{n}:= n^{-1}\sum_{k=1}^n W_{n,t_k}$, 
\begin{eqnarray*}
\widehat{\xi}_{n, t_k}^\star(x,y) &:=& \sqrt{\frac{F_{x,n}(h)}{n}}\,\zeta_{t_k}\left[Y_{t_k}-\widehat{m}_{\psi,n}(x,y)  \right]\frac{\Delta_{t_k}(x)}{\mathbb{E}(\Delta_{t_1}(x))}\\
&& -
	\mathbb{E}\left\{
	\sqrt{\frac{F_{x,n}(h)}{n}}\,\zeta_{t_k}\left[Y_{t_k}-\widehat{m}_{\psi,n}(x,y)  \right]\frac{\Delta_{t_k}(x)}{\mathbb{E}(\Delta_{t_1}(x))} \mid {\cal F}_{k-1}
	\right\},
\end{eqnarray*}

and $\{W_{n,t_k}: k=1, \dots, n\}$ the bootstrap weights generated at the instants $(t_k)_{k=1, \dots, n}$, such that,   for $n\geq 1$, the sequences of random variables
$\{\widehat{\xi}_{n, t_k}(x,y): k=1, \dots, n\}$ and
$\{W_{n,t_k}: k=1, \dots, n\}$ are independent.

In practice the bootstrap procedure can be summarized as follows:

\begin{itemize}
	\item  Suppose
	$
	\left(W^{(1)}_{ n,t_1},\ldots, W^{(1)}_{n,t_n}\right),\ldots,\left(W^{(B)}_{ n,t_1},\ldots, W^{(B)}_{n,t_n}\right)$
	are $B$ independent vectors, sampled independently from the sample $\{ \widehat{\xi}_{n, t_k}^{*}(x,y):k=1, \dots, n\}$.
	\item Consider the random variables
	$$
	\widehat{S^\star_n}^{(\ell)}(x,y)= \left|\sqrt{n}\left\{\sum_{k=1}^{n}\left(W_{n,t_k}^{(\ell)}-\overline{W}_{n}^{(\ell)}\right)  \widehat{\xi}_{n, t_k}^{*}(x,y) \right\}\right|,~~\ell=1,\ldots,B.
	$$
	\item Using Proposition \ref{poropBoostrab} and Remark \ref{remarkBoostrab}, one can use the smallest $z > 0$ as an approximation of the unknown true quantile
	$c_{1-\alpha/2}$,  such that
	$$
	\frac{1}{B}\sum_{\ell=1}^{B}\1_{\left\{\widehat{S^\star_n}^{(\ell)}(x,y)\leq z\right\}}\geq 1-\alpha.
	$$	
\end{itemize}

\section{Simulation study}\label{sec4}
This section aims to discuss numerically some aspects related to continuous time processes that might affect the quality of estimation of the operator $m_\psi(\cdot,y)$.  In this section we consider $\psi_y(Y_t) = Y_t$ and therefore $m_\psi(x,y)$  being   the conditional expectation $m(x)$ of $Y_t$ given $X_t=x.$ The first simulation aims to compare the quality of estimation of $m(x)$ based on a continuous time and discrete time processes. In simulation 2 we discuss the choice of the ``optimal" sampling mesh $\delta$ in the case of continuous time processes and assess its sensitivity towards the missing at random mechanism.
\subsection{Simulation 1: continus-time versus discrete-time estimators}
In this first simulation we try to compare the estimation of the regression operator when discrete and continuous time processes are considered. We want whether or not considering a continuous time processes may improve the quality of the predictions. We suppose that the functional space ${\cal E} = L^2([-1,1])$ endowed with its natural norm. The generation of continuous time processes $(\{X_t(s): s \in [-1,1]\}, Y_t)_{t\in [0,T]}$ is obtained by considering the following steps:
\begin{enumerate}
\item First, we simulate an Ornstein-Uhlenbeck (OU) process $(Z_t)_{t\geq 0}$ solution of the following stochastic differential equation
\begin{eqnarray}\label{OUP}
dZ_t = 2(5-Z_t)dt + 7d\mbox{W}_t,
\end{eqnarray}
where $\mbox{W}_t$ denotes a Wiener process. Here, we take $dt=0.005.$ \item Let $\Gamma(\cdot)$ be the operator mapping $\mathbb{R}$ into $L^2([-1,1])$ defined, for any $z\in \mathbb{R}$, as follows:
$$
\Gamma(z):= (1+\floor{z} - z ) \mbox{P}_{\mbox{num}(\floor{z})} + (z-\floor{z}) \mbox{P}_{\mbox{num}(\floor{z+1})},
$$
where $P_j$ is the Legendre polynomials of degree $j$ and $\mbox{num}(z) := 1+2\,z\,\mbox{sign}(z) - \mbox{sign}(z) (1+\mbox{sign}(z))/2$ and $\floor{\cdot}$ denotes the floor function. 
\item We consider that curves are sampled at 400 equispaced values in [-1,1] and defined, for any $t\in [0,T]$, as
$$
X_t(s) = \Gamma(Z_t)(s), \quad\quad s\in [-1,1].
$$
\item To generate the real-valued process $(Y_t)_{t\in [0,T]}$, the following nonlinear functional regression model is considered:
\begin{eqnarray}
Y_t = m(X_t) + \varepsilon_t,
\end{eqnarray}
where $m(x):= \int_{-1}^1 x^2(t^\prime) dt^\prime$ and $\varepsilon_t = U_t - U_{t-1}$ where $U_t$ is a Wiener process independent of $X_t$. 
\end{enumerate}
Observe that the OU process $\{Z_t: t\in [0,T]\}$ is a real-valued continuous time process (since $dt$ tends to zero). The operator $\Gamma(\cdot)$ has a role to transform each observation in the process $Z_t$ into a curve through the Legendre polynomials. In such way the functional variable $X$ is generated continuously as is the process $(Z_t)$. Moreover, notice that steps 1, 2 and 3 are devoted to simulate the continuous time functional process $\{X_t(s): s\in [-1,1]\}_{t\in [0,T]}$, whereas in step 4 the real-valued continuous time process $(Y_t)_{t\in [0,T]}$ is generated. A sample of 20 simulated curves is displayed in Figure \ref{Xsample} and an example of the real-valued process $(Y_t)$ is given in Figure \ref{Y}. 

\begin{figure}[h]
\begin{center}
\includegraphics[width=9cm,height=5cm]{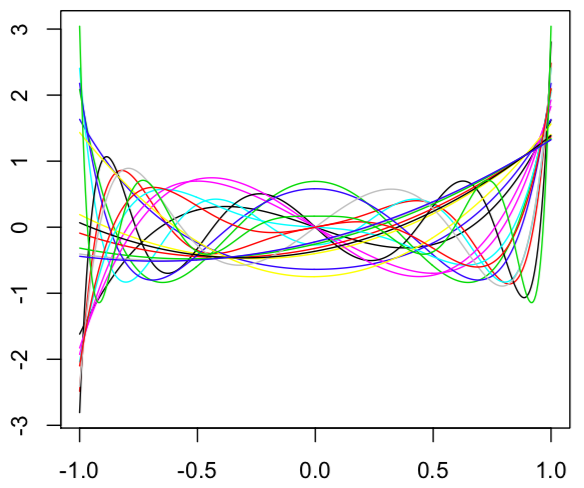}
\vspace{-0.3cm} \caption{\label{Xsample}
 A sample of 20 simulated curves $\{X_t(s) : s\in [-1,1]\}$.}
\end{center}
\end{figure}

\begin{figure}[h]
\begin{center}
\includegraphics[width=9cm,height=6cm]{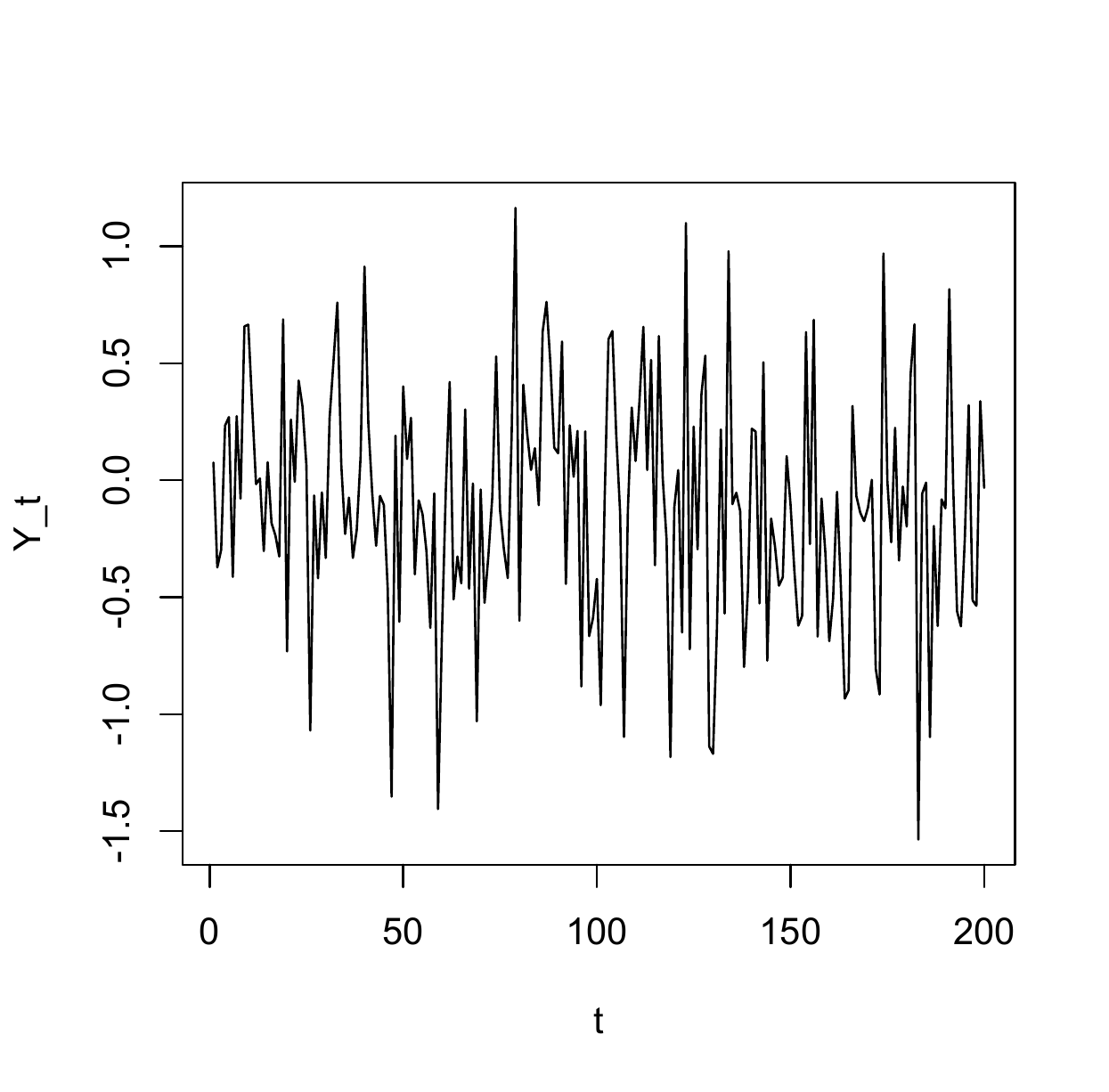}
\vspace{-0.3cm} \caption{\label{Y}
 A realisation of the process $(Y_t)_{t\in [0,200]}$.}
\end{center}
\end{figure}

Now, our purpose is to compare, in terms of estimation accuracy, the continuous-time estimator with the discrete-time one for different values of $T= 50, 200, 1000$ and several missing at random rates.  It is worth noting that the continuous time process $(X_t, Y_t)$ is observed at every instant $t=\delta, 2\delta, \dots, n\delta$, where $\delta=0.005$ and $n=T/\delta$. However, the discret-time process is observed only at the instants $t=1, 2, \dots, n.$

As in \cite{F13} and \cite{Ling15} we consider that the missing at random mechanism is led by the following probability distribution:
\begin{eqnarray}\label{marm}
p(x) = \mathbb{P}(\zeta_t = 1 | X_t =x) = \mbox{expit}\left(\int_{-1}^1 x^2(s) ds\right),
\end{eqnarray}
where $\mbox{expit}(u) = e^u/(1+ e^u)$, for $u \in \R.$ Now, we specify the tuning parameters on which depends our estimation given in \eqref{estimator}. We choose the quadratic kernel defined as $K(u) = \frac{3}{4}(1-u^2)\1_{[0,1]}(u)$ and because curves are smooth enough we choose as semi-metric the $L_2$-norm of the second derivatives of the curves, that is for $t_1 \neq t_2$, $d(X_{t_1}, X_{t_2}) = \left( \int_{-1}^1 [X_{t_1}^{(2)}(s) - X_{t_2}^{(2)}(s)]^2 ds \right)^{1/2}.$ We used the {\it local cross-validation} method on the $\kappa$-nearest neighbours introduced in \cite{F06} page 116 to select the optimal bandwidth for both discrete and continuous time regression estimators. The accuracy of the discrete and the continuous time regression estimators is evaluated on $M=500$ replications. The accuracy is measured, at each replication $j=1, \dots, M$, by using the squared errors $\mbox{SE}_{T}^j := \left(\widehat{m}_T^j(x) - m(x)  \right)^2$ and $\mbox{SE}_n^j := \left(\widehat{m}_n^j(x) - m(x)  \right)^2$ for the continuous-time and the discrete-time estimators, respectively. Observe that the discrete time estimator of the regression operator is defined as:
$$
\widehat{m}_n(x):= \dfrac{\sum_{t=1}^n  \zeta_t Y_t \Delta_t(x)}{\sum_{t=1}^n \zeta_t \Delta_t(x)}.
$$

   \begin{table}[h!]
  \caption{\label{tab1}{Summary statistics of $(\mbox{SE}^j)_{j=1, \dots, 500 }$ for discrete and continuous time estimators of the regression function}.
 }
\begin{tabular}{l|l|lcc|cccc|}

\hline
MAR rate&&&Continuous ($\times 10^{-2}$)&& &Discrete ($\times 10^{-2}$)\\
\cline{1-8}
   &$T$ & 50 &  200 & 1000 & 50&  200& 1000 \\
\cline{2-8}
& $Q_{25\%}$& 0.56 &  0.54&0.18 &1.6  & 1.3 &  0.7\\
$p=20\%$  &Median & 2.4 & 2.58 & 1.25 & 7.5 & 3.7 & 2.1 \\
&Mean & 5.39 & 4.62 & 4 & 15.11 & 11.8 & 9.1 \\
&$Q_{75\%}$ & 6.21 & 7.13 & 4.7 & 16.06 & 10.9 & 4.7 \\
\cline{1-8}
& $Q_{25\%}$ & 0.49 & 0.6 & 0.3& 1.7 & 1& 0.5 \\
$p=40\%$  &Median  & 4.77 & 2.7 & 2.2 & 4.9 & 6.1 & 2 \\
&Mean & 6.93 & 8.8 & 4.6 &12.4 & 9.9 & 8.3 \\
&$Q_{75\%}$ & 9.39 & 10.6 & 4.5 & 16.05 & 11.3 & 11.1 \\
\cline{1-8}
  \end{tabular}


  \end{table}

Table \ref{tab1} shows that continuous time regression estimator is more accurate than the discrete time one. Moreover, when $T$ increases the squared errors decrease much more quickly when working with the continuous time process.

\subsection{Simulation 2: optimal sampling mesh selection}
The purpose of this simulation is to investigate another aspect related to continuous time processes. The selection of the ``optimal'' sampling mesh is one of the most important topics in continuous time processes.  

First of all we generate a continuous-time functional data process according to the following equation:
$$
X_t(s) = Z_t \left(1-\sin(s-\pi/3) \right), \quad\quad s\in [0,\pi/3] \quad \mbox{and}\quad t\in[0,T],
$$ 
where $Z_t$ is an OU process solution of the stochastic differential equation \eqref{OUP} and practically observed at the instants $t=\delta, 2\delta, \dots, n\delta$ with $n=200$ fixed. Here, we take different values of sampling mesh $\delta$, and we study the accuracy of the estimator to identify the optimal mesh, say $\delta^\star$, that minimises the Mean Integrated Square Error (MISE).  Observe that each curve observed at the instant $t$ is discretized in 100 equidistant points over the interval $[0,\pi/3].$ The response variable is obtained following the hereafter nonlinear functional regression model:
$$
Y_t = m(X_t) +\varepsilon_t,
$$
where the operator $m(\cdot)$ is defined as $m(X_t) = \left(\int_{0}^{\pi/3} X_t^\prime(s) ds \right)^2$ and $\varepsilon_t\sim N(0, 0.075).$

Moreover, the missing at random mechanism in this simulation is also supposed to be the same as described in the first simulation as per equation \eqref{marm}. For the tuning parameters used to build the estimator, we considered the quadratic kernel and given the shape of the true regression operator, which depends on the first derivative of the functional covariate, the Euclidean distance between the first order derivatives of the curves is adopted as a semi-metric. Finally the bandwidth is selected according to the local cross-validation method based on the $\kappa$-nearest neighbours as detailed in \cite{F06} page 116. 

For each value of sampling mesh $\delta$, the regression operator $m(\cdot)$ is estimated over a grid of 50 different fixed curves $x$ and the whole procedure is repeated over $N=500$ replications. Finally the MISE is calculated, for each value of $\delta$, according to the following equation
$$
MISE(\delta) := \dfrac{1}{N}\sum_{k=1}^N \dfrac{1}{50}\sum_{j=1}^{50} \left(m(X_j^{(k)}) - \widehat{m}_\delta(X_j^{(k)}) \right)^2.
$$
Observe that $\widehat{m}_\delta(\cdot) $ the estimator of $m(\cdot)$ depends on the sampling mesh $\delta$, so is the MISE. 
\begin{figure}[h!]
\begin{center}
\includegraphics[width=12cm,height=9cm]{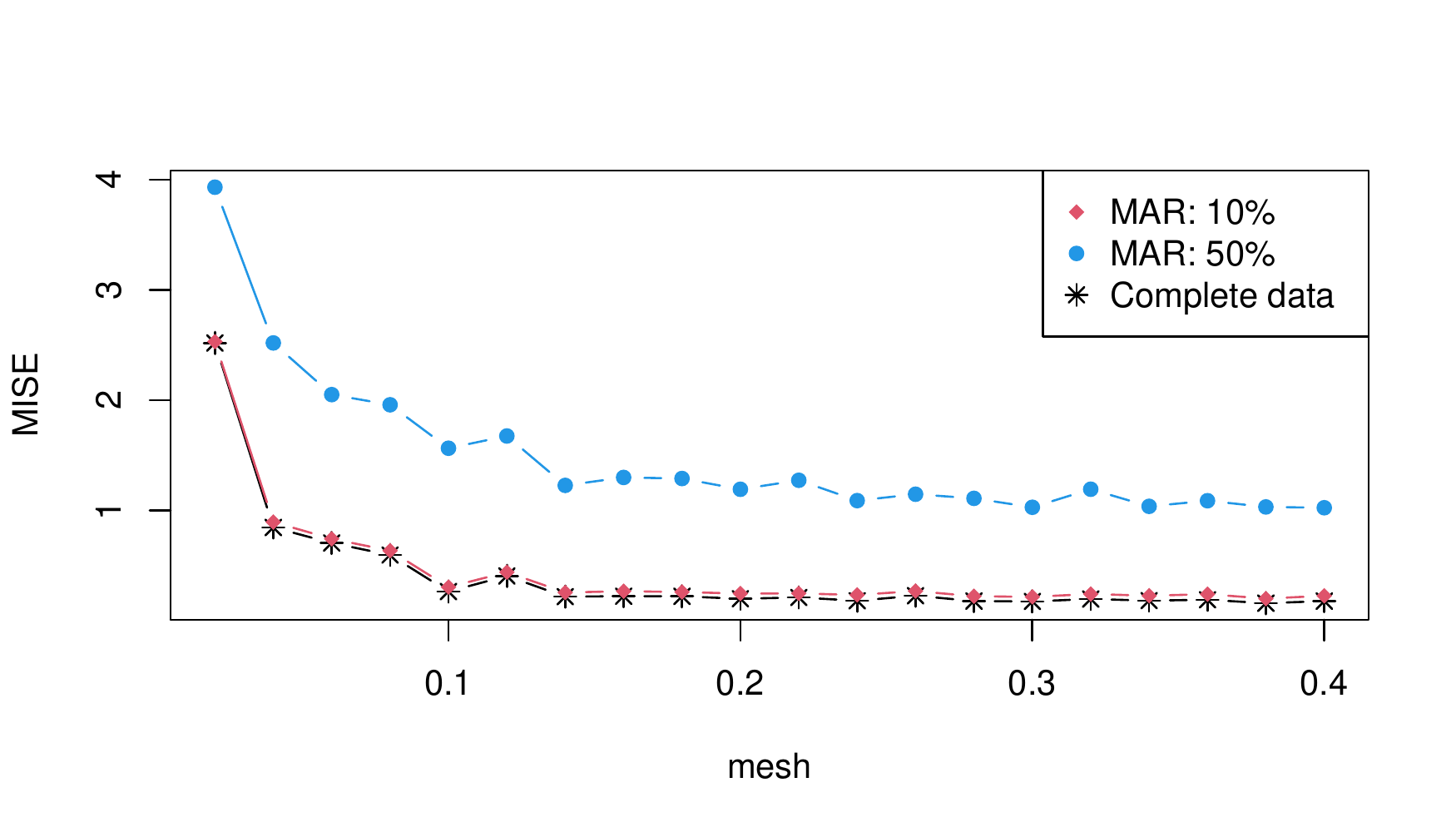}
\vspace{-0.3cm} \caption{\label{MISE}
 The $MISE(\delta)$ obtained for different values of sampling mesh $\delta$ and several missing at random rates.}
\end{center}
\end{figure}
Figure \ref{MISE} displays the values of MISE obtained for different values of sampling mesh $\delta$ and a missing at random rate of 10\%, 50\% and 0\% (complete data), respectively. One can observe that higher is the missing at random rate, higher will be errors in estimating the regression operator. 
\begin{table}[h]
\centering
\begin{tabular}{ |c|c|c|c| c|c|c|} 
\hline
& $\delta^\star$ & MISE($\delta^\star$) & $\overline{MISE(\delta)}$ & $Q_{25\%}$ & $Q_{50\%}$ & $Q_{75\%}$\\
\hline
Complete data & 0.30 & 0.0476 & 0.0562 & 0.0510 & 0.0524 & 0.0550\\
\hline
MAR=10\% & 0.36 & 0.0555 & 0.0643 & 0.0587 & 0.0612 & 0.0635\\
\hline
MAR=50\% & 0.38 & 0.0635 & 0.0767 & 0.0714 & 0.0750 & 0.0770\\
\hline
\end{tabular}
\caption{\label{tab2}The optimal sampling mesh, $\delta^\star$, obtained for different MAR rates and some summary statistics of the MISE($\delta$).}
\end{table}

Table \ref{tab2} shows that higher is the missing at random rate longer we need to observe the underlying process to collect the $n=200$ observations to be able to reasonably estimate the regression operator. Indeed, when the data is complete the optimal time interval $T^\star = n\delta^\star = 200\times 0.3 = 60$. However, when $MAR=10\%$ (resp. 50\%) the optimal time interval is equal to $T^\star = 200\times 0.36 = 72$ (resp. $T^\star = 200\times 0.38 = 76$). Consequently, it can be concluded that when the missing at random mechanism is heavily affecting the response variable, we need to collect data over a longer period of time. This allows to get sufficient information about the dynamic of the underlying continuous time process and therefore get a better estimate of the regression operator.

\section{Application to conditional quantiles}\label{sec5}

Let  $x\in {\cal E}$ be fixed and $y\in \mathbb{R}$,  then if $\psi_y(Y)=\1_{ \left\{ ]-\infty,  \ y] \right\}}(Y)$
 the operator $m_\psi(x,y)$ is no more but the  conditional cumulative distribution function (df) of $Y$ given $X = x$, namely  
 $\ F(y|x)= \mathbb{P}(Y \leq y | X=x)$
which may be estimated by   
$\widehat{F}_T(y |x):=\widehat{m}_{\psi,T}(x,y).$  
For a given $\alpha \in(0, 1),$  the $\alpha^{th}$-order conditional quantile of the distribution of $Y$ given $X = x$ is defined as 
$q_\alpha(x)=\inf\{y \in \mathbb{R}: \ F(y| x)  \geq \alpha \}.$

Notice that, whenever   $F (\cdot | x )$ is strictly
increasing and continuous in a neighbourhood of  $q_\alpha(x)$, the function 
$F (\cdot | x )$ has a unique quantile of order $\alpha$ at a point 
$q_\alpha(x)$, that is $F(q_\alpha(x) | x)=\alpha.$
In such case 
$$q_\alpha(x)=F^{-1} (\alpha | x)=\inf\{y \in \mathbb{R}: \ F(y| x)  \geq \alpha \},$$
which may be estimated uniquely  by 
$\widehat{q}_{T,\alpha}(x)=\widehat{F}_{T}^{-1}(\alpha| x)$.
Conditional quantiles has been widely studied in the literature when  the predictor $X$ is of finite dimension, see for instance,  \cite{G2003} and  \cite{F-al05} for  dependent functional data.

\begin{itemize}
\item[$(a)$] {\bf Almost sure pointwise and uniform convergence}
\end{itemize}

 Under the same conditions of Theorem \ref{thm1}, the statement (\ref{Th-PointwiseConv})  still holds for the estimator of the cumulative conditional distribution function $\widehat{F}_T(y|x)$. That is $\widehat{F}_T(\alpha|x)$ converges, almost surely, towards $F(y|x)$ with a rate $\OH(h_T^\beta) + \OH(\sqrt{\log T/(T\phi(h_T))}).$
 
 Consequently, since $F(q_\alpha(x) | x)=\alpha= \widehat{F}_T(\widehat{q}_{T,\alpha}(x) | x)$ and   $\widehat{F}_T (\cdot | x )$  is continuous and strictly increasing, then we have $\forall \epsilon>0, \ \exists \eta(\epsilon)>0, \  \forall y, \quad 
 \big| \widehat{F}_T(y | x)- \widehat{F}_T(q_\alpha(x) | x) \big| \leq 
  \eta(\epsilon) \quad \Rightarrow \quad | y-  q_\alpha(x) | \leq  \epsilon$
 which implies that, $ \forall \epsilon>0, \ \exists \eta(\epsilon)>0,$ 
 \begin{eqnarray}\label{convQuantile}
\mathbb{P}\left( \left|    \widehat{q}_{T,\alpha}(x) - q_\alpha(x)             \right|\geq \eta(\epsilon) \right)
  \leq 
 \mathbb{P}\left( \left|
 \widehat{F}_T(\widehat{q}_{T,\alpha}(x) | x)-\widehat{F}_T(q_\alpha(x) | x)\right|
 \geq \eta(\epsilon) \right) \nonumber\\
 =
 \mathbb{P}\left( |
 F(q_\alpha(x) | x)-\widehat{F}_T(q_\alpha(x) | x)
 \geq \eta(\epsilon) \right).
 \end{eqnarray}
 Therefore, the statement 
 (\ref{Th-PointwiseConv})  still holds for the  conditional quantile estimator  
 $\widehat{q}_{T,\alpha}(x)$ whenever conditions of Theorem 
 \ref{thm1} are satisfied.  \cite{F-al05} derived similar pointwise convergence rate by  inverting the estimator of  the conditional cumulative distribution function. Their result has been  obtained under  mixing condition and additional assumptions on the  joint distribution,  and the Lipschitz condition upon $F(y |x)$ and its derivatives with respect $y$. 
	
Regarding the almost sure uniform convergence, observe that under  conditions of Theorem \ref{uniform}, the statement   (\ref{V-unifome1}) still holds  true   for the 
$\sup_{y\in S}\sup_{x \in {\cal C}} |\widehat{F}_T(y(x) - F_T(y(x)|$,  
 when  $\psi_y(Y)$   is  replaced by $\mathbf{1}_{ \{ ]-\infty,  \ y] \}}(Y)$.  Moreover,  assume that, for fixed $x_0 \in {\cal C}$, $F( y | x_0)$ is differentiable at $q_\alpha(x_0)$ with
$\frac{\partial }{\partial y}F( y | x_0)|_{y=q_\alpha(x_0)}:=g(q_\alpha(x_0) | x_0)>\nu>0$, and $g(\cdot | x)$ is  uniformly continuous  for all $x \in {\cal C}$.
 Knowing that $\widehat{F}_T(\widehat{q}_{T, \alpha}( x )| x) = F(q_{\alpha}( x )| x )=\alpha$ and making use of a Taylor's expansion of the function $F(\widehat{q}_{T, \alpha}( x )| x )$ around $q_\alpha( x )$, we can write 

\vskip -6mm
\begin{eqnarray}\label{VitesseQuntile}
F(\widehat{q}_{T, \alpha}( x )| x )- F(q_{\alpha}( x )| x )=
\left(\widehat{q}_{T, \alpha}( x )-  q_{\alpha}( x ) \right)g\left(q^*_{T, \alpha}( x )| x \right)
\end{eqnarray}
where $q^*_{T, \alpha}( x )$ lies between $q_{\alpha}( x )$ and 
$\widehat{q}_{T, \alpha}( x )$. It follows  then from (\ref{VitesseQuntile}) that  the inequality  (\ref{convQuantile})   
still hods true uniformly in $x$ and $y$. Moreover, the fact that 
$\widehat{q}_{T, \alpha}( x )$ converges (a.s.) towards $q_{\alpha}( x )$ as $T$ goes to infinity, combined with the uniformly continuity of $g(\cdot | x)$, allow to write that
%
\begin{eqnarray}\label{VitesseQuntile2}
\sup_{x \in {\cal C}}\big|\widehat{q}_{T, \alpha}( x )- q_{\alpha}( x )\big|
\sup_{x \in {\cal C}}\big|g\left(q_{\alpha}( x )| x \right)\big|= 
O_{a.s.}\left(\sup_{y\in S}\sup_{x \in {\cal C}} |\widehat{F}_T(y|x) - F(y|x)|  \right).
\end{eqnarray}
Since  $g(q_{\alpha}( x ) |x)$  is uniformly bounded from below, we can then clame that the estimator $\widehat{q}_{T, \alpha}( x )$ converges uniformly  towards $q_\alpha(x)$  with the same convergence rate given in (\ref{V-unifome1}), as $T$ goes to infinity. 

\begin{itemize}
\item[$(b)$] {\bf Confidence intervals}
\end{itemize}

Confidence bounds for the conditional quantiles $q_\alpha(x)$  may be obtained  according to the following steps. First, consider a Taylor's expansion of $\widehat{F}_T(\cdot|x)$ around $q_\alpha(x)$ and making use of the fact that 
$\widehat{q}_{T, \alpha}( x )$ converges (a.s.) towards $q_{\alpha}( x )$ as $T$ goes to infinity, one gets
%
\begin{eqnarray}\label{VitesseQuntile2-0}
\widehat{q}_{T, \alpha}( x )- q_{\alpha}( x )
=- \frac1{\widehat{g}_T\left(q_{\alpha}( x )| x \right)}
 \left(\widehat{F}_T(q_\alpha(x)|x) - F(q_\alpha(x)|x)  \right),
\end{eqnarray}
where $\widehat{g}_T(\cdot|x)$ is a consistent estimator of $g(\cdot|x)$. Then, replacing 
$\psi_y(Y)$ by  the indicator function, we get under  
conditions of Corollary \ref{colorlarryNor},  the following 
  $(1-\alpha)100\%$ confidence bands for  $q_\alpha(x)$ 
\begin{eqnarray}\label{Corollay-intr}
\widehat{q}_{T, \alpha}( x )\  \pm c_{1-\alpha/2} \frac{M_{T,1} \widehat{g}_T\left(\widehat{q}_{T,\alpha}( x )| x \right)}{\sqrt{M_{T,2}}}\sqrt{\frac{\alpha(1-\alpha)}{T F_{x,T}(h_T) p_T(x)}}, \quad \text{as} \ T\to \infty.
\end{eqnarray}
%

%
%

\section{Proofs}\label{secProofs}
%

In this section and for sake of simplification the subscript $\psi$ will be omitted. Consider the following quantities 
\begin{eqnarray}
Q_T(x,y) &:=& (\widehat{m}_{T,2}(x,y) - \overline{m}_{T,2}(x,y)) - m(x,y) (\widehat{m}_{T,1}(x,y) - \overline{m}_{T,1}(x,y)) \label{Q},\\
R_T(x,y) &:=& - B_T(x,y) (\widehat{m}_{T,1}(x,y) - \overline{m}_{T,1}(x,y)) \label{R}.
\end{eqnarray}
We have then 
\begin{eqnarray}\label{decomp}
\widehat{m}_T(x,y) - m(x,y) = \widehat{m}_T(x,y) - C_T(x,y) + B_T(x,y) = B_T(x,y) + \frac{Q_T(x,y) + R_T(x,y)}{\widehat{m}_{T,1}(x,y)},
\end{eqnarray}
We start first by stating  some  technical lemmas that will be used later.
\begin{lemma}\label{lem1}
Assume that assumptions (A1)-(A2) are satisfied, then we have for any $j\geq 1$ and $\ell\geq 1$
\begin{itemize}
\item[$(i)$] $\frac{1}{\phi(h_T)} \E\left(\Delta_{t}^\ell (x) | \F_{j-2} \right) = M_{\ell} f_{t, T_{j-2}}(x) + \mathcal{O}_{a.s.}\left(\frac{g_{t, T_{j-2}, x}(h_T)}{\phi(h_T)} \right),$
\item[$(ii)$]  $\frac{1}{\phi(h_T)} \E\left(\Delta_{t}^\ell(x) \right) = M_\ell f(x) + o(1)$.
\end{itemize}
\end{lemma}
\begin{proof}  The proof is similar to the proof of Lemma  of \cite{LL10}.
\end{proof}		
\begin{lemma}\label{lem2}
Let $(Z_n)_{n\geq 1}$ be a sequence of real martingale differences with respect to the sequence of $\sigma$-fields $\left( \F_n = \sigma(Z_1, \dots Z_n)\right)_{n\geq 1}$ where $\sigma(Z_1, \dots, Z_n)$ is the $sigma$-field generated by the random variable $Z_1, \dots, Z_n$. Set $S_n = \sum_{i=1}^n Z_i.$ For any $p\geq 2$ and any $n\geq 1$, assume that there exist some nonnegative constants $C$ and $d_n$ such that $\E\left(Z_n^p | \F_{n-1} \right) \leq C^{p-2} p! d_n^2$ almost surely. Then, for any $\epsilon >0,$ we have
$$
\mathbb{P}\left(|S_n| > \epsilon \right) \leq 2 \exp\left\{-\frac{\epsilon^2}{2(D_n + C\epsilon)} \right\},
$$
where $D_n = \sum_{i=1}^n d_i^2.$
\end{lemma}

\begin{proof}
See Theorem 8.2.2  of  \cite{PG99}.
\end{proof}


%
%
%
%
%
%
\vskip 2mm
\noindent {\bf Proof of Theorem \ref{thm1}.}  The proof of Theorem \ref{thm1} is a consequence of decomposition (\ref{Q}) and  Lemmas \ref{lem3}-\ref{mhat}-\ref{BR} established below. $\hfil \Box$
%
\begin{lemma}\label{lem3} Assume  (A1)-(A3) tougher with   conditions (\ref{cond}) hold true.  
	Then, we have  for $T$ sufficiently large
	\begin{eqnarray*}
		\widehat{m}_{T,j}(x,y) - \overline{m}_{T,j}(x,y) = \OH_{a.s.}\left(\sqrt{\frac{\log T}{T\phi(h_T)}} \right) \quad \text{with} \quad j\in \{1, 2\}.
	\end{eqnarray*}
\end{lemma}
\noindent {\bf   Proof of Lemma \ref{lem3} }.   Consider the case where $j=2$ and define the process 
	\begin{eqnarray*}
	L_T:=	\widehat{m}_{T,2}(x,y) - \overline{m}_{T,2}(x,y) 
		&=& \frac{1}{n\E(Z_1(x))} \sum_{j=1}^n \int_{T_{j-1}}^{T_j} \left[ \zeta_t \psi_y(Y_t) \Delta_t(x) - \E\left(\zeta_t \psi_y(Y_t) \Delta_t(x) | \F_{t-\delta} \right) \right] dt\\
		&=:& \frac{1}{n\E(Z_1(x))} \sum_{j=1}^n \mathcal{L}_{T,j}(x,y),
	\end{eqnarray*}
Observe that, 
for any $j\geq 1$ and $t\in[ T_{j-1}, T_j]$, ${\cal F}_{j-2}\subset {\cal F}_{t-\delta} \subset {\cal F}_{j-1}$, therefore
$L_{T,j}(x,y)$ is ${\cal F}_{j-1}$-measurable,
$\mathbb{E}\left( |L_{T,j}(x,y)|\right)<\infty$ provided $\mathbb{E}(\zeta^2_t)<\infty$ and $\mathbb{E}(X^2_t)<\infty$
(in view of Cauchy-schwartz inequality). Moreover,   letting
$\eta_j:=\int_{T_{j-1}}^{T_j} \left[\zeta_t (\psi_y(Y_t) \Delta_t(x)\right]dt,$
then
$L_{T,j}(x,y)=\eta_j-\mathbb{E}[\eta_j | {\cal F}_{t-\delta}]$
and 

$\mathbb{E}\left\{L_{T,j}(x,y) | {\cal F}_{j-2}\right\}=
\mathbb{E}\left\{\mathbb{E}\left[ \eta_j| {\cal F}_{t-\delta}\right]|  {\cal F}_{j-2}\right\}- \mathbb{E}\left\{\mathbb{E}\left[ \eta_j| {\cal F}_{t-\delta}\right]|  {\cal F}_{j-2}\right\}=0.$

Hence  $(L_{T,j}(x,y))_{j\geq 1}$ is a sequence of martingale differences   with respect to the family  $({\cal F}_{j-1})_{j\geq 1}$. 
%
To be able to apply Lemma \ref{lem2} 
 we need first to check its  conditions. 	
	Applying Jensen and Minkowski inequalities, we get, for any $\kappa\geq 2$, that 
	\begin{eqnarray*}
		|\E\left[\mathcal{L}^\kappa_{T,j}(x,y) | \F_{j-2} \right]|
%
	&\leq& 	 2^\kappa \int_{T_{j-1}}^{T_j} \E\left[|\zeta_t \psi_y(Y_t) \Delta_t(x)|^\kappa | \F_{j-2} \right] dt.
	\end{eqnarray*}
	Using  successfully a double conditioning with respect to the $\sigma$-fields $\ST_{t,\delta}^0$ and $\ST_{t,\delta}^1$ combined with   (A3)(i), 
	 (A3)(iii), (A3)(iv), (A3)(i$v^\prime$)   and the fact that $\zeta$ and $Y$ anr conditionally independent given $X$,  one gets  for any $t\in [T_{j-1}, \   T_j]$    and any  $p\geq 2$, that 
	\begin{eqnarray*}
		\E\left[|\zeta_t \psi_y(Y_t) \Delta_t(x)|^p | \F_{j-2} \right]
	%
		%
		=\E\left\{ \Delta^p_t(x) U_p(X_t) W_p(X_t,y)\Big | \F_{j-2}\right\} 
		%
		%
		%
		%
		\leq C_{W,U} \; \E\left(\Delta_t^p(x) | \F_{j-2} \right),
	\end{eqnarray*}
	where $C_{W,P}=\max(\sup_x |U_p(x)|, \sup_{x,y}|W_p(x,y)|)$ which is a finite positive constant independent on $(x,y)$ in view of conditions  (A3)(iii)  and (A3)(i$v^\prime$) .   
	
		Since the kernel $K$ and the function $\tau_0$ are bounded from above by a positive constants $a_1$ and $c_0$ respectively,
		and the  function $f_{t, T_{j-2}}(x)$ is bounded by the deterministic function $b_{t,2\delta}(x)$; using then assumption (A2)(ii)	combined with Lemma \ref{lem1} to get 
	\begin{eqnarray}\label{Moemnts}
		\left|\E\left[\mathcal{L}^\kappa_{T,j}(x,y) | \F_{j-2} \right]\right| &= & 2^\kappa \int_{T_{j-1}}^{T_j} C_{W,U} \phi(h_T) \left[M_\kappa f_{t, T_{j-2}}(x) + \OH\left(\frac{g_{t, T_{j-2}, x}(h_T)}{\phi_y(h_T)} \right) \right] dt \nonumber\\
		%
	&\leq& \kappa! (2a_1)^{\kappa-2}  \phi(h_T)\left[(2a_1)^2 C_{U,W} \int_{T_{j-1}}^{T_j} b_{t,2\delta}(x)   dt + o(1) \right].	
	\end{eqnarray}
Letting  $d_{j-2}^2 := \phi_y(h_T)\left[ (2a_1)^2 C_{U,W} \int_{T_{j-1}}^{T_j} b_{t,2\delta}(x) dt + o(1)\right]$ and define

		$D_n = \frac{1}{n} \sum_{j=2}^n d_{j-2}^2 = \phi(h_T) \left\{(2a_1)^2 C_{W,U} \sum_{j=2}^n \int_{T_{j-1}}^{T_j} b_{t,2\delta}(x) dt +o(1)\right\}$. Since  $T=n\delta$, 
		it follows  from (A2)(v)  that
		$D_n = T\phi(h_T) \left[(2a_1)^2 C_{W,U} D_{2\delta}(x) + o(1) \right].$
	
	Moreover,  we have from the statement (ii) of Lemma \ref{lem1}    that $n\E\left(Z_1(x) \right) = \OH(T\phi(h_T))$.  Lemma \ref{lem2} combined with condition (\ref{cond}) allow  to write,  for  any $\epsilon_0>0$, that 
%
	\begin{eqnarray}\label{Pointwise}
		\PR\left(\Big|\widehat{m}_{T,2}(x,y) - \overline{m}_{T,2}(x,y)\Big| > \epsilon_0\sqrt{\frac{\log T}{T\phi(h_T)}}\right) &=& \PR\left(\Big|\sum_{j-1}^n {\cal L}_{T,j}(x,y)\Big| > n\E(Z_1(x)) \epsilon_0 \left(\frac{\log T}{T\phi(h_T)} \right)^{1/2}  \right) \nonumber\\
		%
		&\leq& 2\exp\left\{-c\epsilon_0^2 \log T  \right\} = \frac{2}{T^{c\epsilon_0^2}}
	\end{eqnarray}
	where $c$ is a positive constant. 	Let $T_n:=[T]$ be the integer part of $T$ and  choosing $\epsilon_0$ large  enough,
%
	and  using condition (\ref{cond}) with Borel-Cantelli Lemma to conclude that (\ref{Th-PointwiseConv}) is valid  since $\sum_{n\geq 1}\frac{1}{T_n^{c\epsilon_0}}<\infty$. $\hfill \Box$.

\begin{lemma}\label{mhat}
	Under assumption (A1)-(A3), we have
	$$
	\widehat{m}_{T,1}(x) \longrightarrow p(x), \quad\mbox{a.s.}\quad\quad \mbox{as}\quad T\rightarrow +\infty.
	$$
\end{lemma}
\begin{proof} Let us introduce the	following decomposition
	\begin{eqnarray}\label{Decpmaht}
\widehat{m}_{T,1}(x)-p(x)=\left(\widehat{m}_{T,1}(x) - \overline{m}_{T,1}(x)\right) + \left(\overline{m}_{T,1}(x) - p(x)\right) =: {\cal M}_{T,1}(x) + {\cal M}_{T,2}(x) 	
	\end{eqnarray}
 We have  from  Lemma \ref{lem3} that
	\begin{eqnarray}\label{M1}
	{\cal M}_{T,1} = \OH_{a.s.}\left(\left(\frac{\log T}{T\phi(h_T)} \right)^{1/2} \right)=o_{a.s}(1).
	\end{eqnarray}
	Let us now focus on the second term ${\cal M}_{T,2}$.  Using   a double conditioning with respect to  the $\sigma$-field ${\cal S}_{t-\delta, \delta}^1$, assumptions (A2)(ii)-(iii),  (A3)(iv)-(i$v^\prime$) and Lemma \ref{lem1} one gets
	\begin{eqnarray}\label{m1-bar}
		\overline{m}_{T,1}(x) 
		&=& \frac{1}{n\E(Z_1(x))} \int_0^T \E\left[\Delta_t(x) p(X_t) | \F_{t-\delta} \right] dt \nonumber\\
		&=& \frac{1}{n\E(Z_1(x))} \left(p(x) + o(1) \right) T \phi(h_T) \left\{M_1 \frac{1}{T}\int_0^T f_{t, t-\delta}(x) dt + \frac{1}{T} \int_0^T \OH\left(\frac{g_{t, t-\delta,x}(h_T)}{\phi(h_T)} \right) dt \right\} \nonumber \\
		&=& \left(p(x) + o(1) \right) \frac{1}{n\E(Z_1(x))}\; \OH(T\phi(h_T))= p(x)+o(1).
	\end{eqnarray}
	Thus  $\overline{m}_{T,1}(x)$ converges almost surely towards  $p(x)$ as $T\rightarrow +\infty$, which  achieve the proof of this lemma.
\end{proof}
\begin{lemma}\label{BR}
	Under hypotheses (A3)(i)-(ii), for a fixed $x\in {\cal E}$, we have
	\begin{eqnarray}\label{TermesBT-RT}
		B_T(x,y) = \OH_{a.s.}(h_T^\beta), \quad 
	%
	\end{eqnarray}
	\begin{eqnarray}\label{QT}
		R_T(x,y) = \OH_{a.s.}\left(h_T^\beta \left(\frac{\log T}{T\phi(h_T)}\right)^{1/2} \right) \quad \text{and} \quad 	
Q_T(x,y) = \OH_{a.s.}\left(\left( \frac{\log T}{T\phi(h_T)} \right)^{1/2} \right).
\end{eqnarray}	
\end{lemma}
\begin{proof} 
	Observe that
	\begin{eqnarray}\label{btilde}
	B_T(x,y) = \frac{\overline{m}_{T,2}(x,y) - m(x,y) \overline{m}_{T,1}(x)}{\overline{m}_{T,1}(x)} =: \frac{\widetilde{B}_T(x,y)}{\overline{m}_{T,1}(x)}.
	\end{eqnarray}

	Using  (A3)(i)-(ii) and a double conditioning with respect to the $\sigma$-field ${\cal S}_{t-\delta, \delta}^0$, and Lemma \ref{mhat}, we get 
	\begin{eqnarray*}\label{biastilde}
		\left| \widetilde{B}_T(x,y) \right|
		&=& 
		\frac{1}{n\E(Z_1(x))} \int_0^T \E\left(\zeta_t \Delta_t(x) 
		\left| 
		m(X_t,y) - m(x,y)  \right|\;   
		| \F_{t-\delta} \right) dt \\
		&\leq& 
		\sup_{u\in B(x,h_T)}|m(u,y) - m(x,y)| \times \frac{1}{n\E(Z_1(x))} \int_0^T \E\left(\zeta_t \Delta_t(x) | \F_{t-\delta}\right) dt \\
		&=& \OH(h_T^\beta) \times |\overline{m}_{T,1}(x)|
	\end{eqnarray*}
which implies that $B_T(x,y)= \OH(h_T^\beta)$. 
The statements  of (\ref{QT}) follow from (\ref{TermesBT-RT}) and Lemma \ref{lem3}.
\end{proof}


\bigskip
\noindent{\bf Proof of Theorem \ref{MSE}}.
	We have from (\ref{decomp}) and Lemma (\ref{mhat}) that
	\begin{eqnarray}\label{MSE1}
	\textsc{MSE}(x,y) &=& \mathbb{E}\left(\widehat{m}_T(x,y) - m(x)\right)^2  \simeq \mathbb{E}\left(B_T(x,y) + \frac{Q_T(x,y) + R_T(x,y)}{p(x)}\right)^2\nonumber\\
	&\simeq& \mathbb{E}\left(B^2_T(x,y)\right) + \frac{1}{p^2(x)}\left[\mathbb{E}\left(Q^2_T(x,y)\right)+ \mathbb{E}\left(R^2_T(x,y)\right)\right],
	\end{eqnarray}
	where the products $2\mathbb{E}\left[B_T(x,y) (Q_T(x,y)+R_T(x,y))\right]$ and  $2\mathbb{E}\left[Q_T(x,y)\times R_T(x,y)\right]$ have been ignored because by the
	Cauchz-schwarz inequality
	\begin{eqnarray*}
		\mathbb{E}\left[B_T(x,y) (Q_T(x)+R_T(x,y))\right]
		&\leq &\mathbb{E}\left(B^2_T(x,y)\right)^{1/2}\times
		\mathbb{E}\left((Q_T(x,y)+R_T(x,y)\right)^{1/2}
		\nonumber\\
		&\leq &\max \left\{ \mathbb{E}\left(B^2_T(x,y)\right), \
		\mathbb{E}\left[(Q_T(x,y)+R_T(x,y))^2\right]
		\right\}.
	\end{eqnarray*}
	We have the same inequality for the second product. The proof of  Theorem \ref{MSE} results from Theorem \ref{bias} and
	Lemma \ref{esp.Q} below, which 
gives an upper bound of the expectation of $Q_T^2(x,y)$ and  $R^2_T(x,y)$, respectively.
\begin{lemma}\label{esp.Q} Assume that assumptions (A1)-(A3) hold true, then
	\begin{eqnarray}
	\E(Q^2_T(x,y)) \simeq  \frac{4 p(x) (W_2(x,y) + (m(x,y))^2) M_2}{T\phi(h_T) M_1^2 f(x)},
	\end{eqnarray}
	
\end{lemma}

\begin{proof} Ignoring the product term as above, one may write
	$$\mathbb{E}\left[Q^2_T(x,y)\right] \simeq \mathbb{E}\left[\widehat{m}_{T,2}(x,y)- \overline{m}_{T,2}(x,y)\right]^2+
	m^2(x,y)\ \mathbb{E}\left[\widehat{m}_{T,1}(x)- \overline{m}_{T,1}(x)\right]^2:= I_{T1}+m^2(x,y)I_{T2}.$$
	The terms $I_{T1}$ and $I_{T2}$ can be handled  similarly. We will juste evaluate the first one. Since $(T_j=j\delta)_{0\geq j\geq n}$
	is a $\delta$-partition of $[0, T]$, we have
	\begin{eqnarray}\label{trem.L2}
	\widehat{m}_{T,2}(x,y)- \overline{m}_{T,2}(x,y) 
	&=&\frac{1}{n\E(Z_1(x))} \sum_{j=1}^n \int_{T_{j-1}}^{T_j} \left[\zeta_t (\psi_y(Y_t) \Delta_t(x) -
	\mathbb{E}\left\{\zeta_t \psi_y(Y_t) \Delta_t(x)| {\cal F}_{t-\delta}\right|\}\right]dt\nonumber\\
	&:=&\frac{1}{n\E(Z_1(x))} \sum_{j=1}^n  L_{T,j}(x,y).
	\end{eqnarray}
	  
	Since   $(L_{T,j}(x,y))_{j\geq 1}$ is a sequence of martingale differences   with respect to the family  $({\cal F}_{j-1})_{j\geq 1}$, then 
	%
	%
	$\mathbb{E}(L_{T,j}(xy,)L_{T,k}(x,y))=0 \quad \mbox{for every} \quad  j, k\in\{1, \ldots n \} \quad \mbox{ such that} \quad  j\neq k.$
	Therefore (by ignoring the product term), we have
	\begin{eqnarray}
	I_{T1}&=&\mathbb{E}\left[\widehat{m}_{T,2}(x,y)- \overline{m}_{T,2}(x,y)\right]^2
	\simeq  \frac{1}{n^2(\E(Z_1(x)))^2}\ \sum_{j=1}^n \mathbb{E}(L_{T,j}(x,y))^2.
	\end{eqnarray}
	Using  Jensen inequality    and conditioning  successively two times  with respect to  ${\cal S}_{t-\delta, \delta}^0$ and ${\cal S}_{t-\delta, \delta}^1$ combined wit (A3)(iii)-(iv), $I_{T1}$  may bounded  as
	\begin{eqnarray*}
		I_{T1} &\leq&\frac{4}{n^2(\E(Z_1(x)))^2}\ \sum_{j=1}^n \int_{T_{j-1}}^{T_j} \E\left[ \Delta_t^2(x) p(X_t) W_2(X_t,y) \right] dt.\\
		&=& \frac{4(p(x) + o(1))(W_2(x,y) + o(1))[M_2 f(x) + o(1)]}{T\phi(h_T) [M_1 f(x) + o(1)]^2}.
	\end{eqnarray*}
	On the other side, we can easily show that
	\begin{eqnarray*}
		I_{T2} \leq \frac{4(p(x)+o(1)) [M_2 f(x) + o(1)]}{T\phi(h_T) [M_1 f(x) + o(1)]^2}.
	\end{eqnarray*}
	
	Therefore,
	\begin{eqnarray*}
		\E(Q^2_T(x,y)) &\simeq& I_{T1} + m^2(x,y) I_{T2}
		= \frac{4 p(x) (W_2(x,y) + (m(x,y))^2) M_2}{T\phi(h_T) M_1^2 f^2(x)}
	\end{eqnarray*}
	Moreover, using the decomposition (\ref{R}) and Theorem \ref{bias}  and Lemma \ref{lem3}  one can see that
$\E(Q^2_T(x,y))$ is  negligible with respect to $\E(Q_T^2(x,y)).$ This completes the proof.
\end{proof}


\bigskip 
\noindent{\bf Proof of Theorem \ref{AsymNormality}}.  The proof of Theorem \ref{AsymNormality} is based essentially on Lemma \ref{AsymNormalityQ} established below,  which  gives the normality asymptotic of the principal term in the decomposition (\ref{decomp}). Indeed,  using decomposition (\ref{decomp}), one may write 
\begin{eqnarray}\label{decompbis}
\sqrt{T\phi(h_T)}\left( \widehat{m}_T(x,y) - m(x,y)\right) =  \sqrt{T\phi(h_T)}B_T(x,y) + \frac{\sqrt{T\phi(h_T)}Q_T(x,y) + \sqrt{T\phi(h_T)}R_T(x,y)}{\widehat{m}_{T,1}(x)}.
\end{eqnarray}
Lemma (\ref{mhat}) implies,  	under assumption (A1)-(A3), 
that 
$\widehat{m}_{T,1}(x) \longrightarrow p(x) \ a.s.$ as  $T\rightarrow\infty.$
Moreover, using Lemma (\ref{BR}), we get  
under (A3)(i)-(ii) combined  with conditions (\ref{ConTheoremNor}) that
$
\sqrt{T\phi(h_T)}	B_T(x,y) = \OH_{a.s.}(h_T^\beta\sqrt{T\phi(h_T)})=o_{a.s}(1),
$ and 
$$
\sqrt{T\phi(h_T)}	R_T(x,y) = \OH_{a.s.}\left(\sqrt{T\phi(h_T)}h_T^\beta \left(\frac{\log T}{T\phi(h_T)}\right)^{1/2} \right)=0
\left(h_T^\beta \log T^{1/2}\right)=
o_{a.s.}(1)
$$
The proof may be  then  achieved by  Lemma \ref{AsymNormalityQ}  and
 Slutsky's  Theorem. 
$\hfil \Box$



%
\begin{lemma}\label{AsymNormalityQ}
	Under conditions (A1)-(A3), we have 
	$$\sqrt{T\phi(h_T)}(\widehat{m}_T(x,y)-m(x,y)) \quad \xrightarrow{d}\quad {\cal N}(0, \tilde{\sigma}^2(x,y)) \quad \text{where} \quad $$
$$\tilde{\sigma}^2(x,y) \leq \frac{M_2}{M_1^2 f(x)}p(x)\overline{W}_2(x,y) \quad 
\text{as} \quad T\longrightarrow \infty.$$
\end{lemma}
\begin{proof} We have 
\begin{eqnarray}\label{NormalityProofLema1}
\sqrt{T\phi(h_T)}Q_T(x,y)&=&\sum_{i=1}^n \xi_{T,i}(x,y),  \quad \text{with}\quad 
\xi_{T,i}(x,y)=\eta_{T,i}(x,y)-\E\left[\eta_{T,i}(x,y) \ \Big| {\cal F}_{t-\delta} \right] \text{and}  \nonumber\\
\eta_{T,i}(x,y) &=& \frac1{ \E Z_1}\sqrt{\frac{\phi_T(h)}{n}}
\int_{T_{i-1}}^{T_i}\zeta_t\Delta_t(x) \left[\psi_y(Y_t)-m(x,y)\right] dt
\end{eqnarray}

Now observe that for any $i\geq 1$   and $t\in [T_{i-1}, T_i]$
${\cal F}_{i-2}\subset {\cal F}_{t-\delta}\subset {\cal F}_{i-1}$. Therefore $(\xi_{T,i}(x,y))_{i\geq 1}$ is ${\cal F}_{i-1}$-measurable,  and 
$\E(|\xi_{T,i}|)<\infty$ provided $\E(\zeta_t^2)<\infty$ and $\E(X_t^2)<\infty$. Moreover, we have for any $1\leq i\leq n$ that 
$\E\left( \xi_{T,i} \ \Big |  {\cal F}_{i-2}\right)=
\E\left\{\E\left[\eta_i \ | {\cal F}_{t-\delta}  \right]  
\Big |  {\cal F}_{i-2} \right\}-
\E\left\{\E\left[\eta_i \ | {\cal F}_{t-\delta}  \right]  
\Big |  {\cal F}_{i-2} \right\}=0.	$

Hence $(\xi_{T,i}(x,y))_{i \geq 1}$  is a sequence of martingale differences with respect to the $\sigma$-fields 
$( {\cal F}_{i-1})_{i\geq 1}$. 
To prove the asymptotic normality, it suffices to check the
two following  conditions (see, Corollary 3.1, p. 56, \cite{HaHe(1980)}):
 
(a) $\sum_{i=1}^n \mathbb{E}[\xi^2_{T,i}( x,y) | \mathcal{F}_{
		i-2}] \overset{\mathbb{P}}{\rightarrow} \tilde{\sigma}^2( x,y)$ 
	
	and (b)
	 $n \mathbb{E}[\xi_{T,i}^2( x,y) \mathbf{I}_{\{|\xi_{T,i}(x,y)| > \epsilon \}}] =
	o(1)$ holds, for any $\epsilon >0$,
	
%

\vskip 1mm
\noindent{\bf Proof of (a)}. Observe now that
$$\Big | \sum_{i=1}^n \E \left[\eta^2_{T,i}(x,y) \ \Big | {\cal F}_{i-2} \right]- 
\sum_{i=1}^n \E \left[\xi^2_{T,i}(x,y) \ \Big | {\cal F}_{i-2} \right]\Big| \leq 
\Big| \sum_{i=1}^n \left(\E \left[\eta_{T,i}(x,y) \ \Big | {\cal F}_{i-2} \right]\right)^2\Big |
$$
Using (A1),  (A3)(i)-(iv) with  Lemma \ref{lem1}, and conditioning two times with respect  to the $\sigma$-field ${\cal S}^{1}_{t-\delta, \delta}$  and 
the fact that $n\mathbb{E} Z_1(t)=O(T\phi(h))$, we have
\begin{eqnarray*}\label{NormalityProofLema1}
&&	\Big| \E\left(\eta_{T,i} \Big| {\cal F}_{i-2}\right) \Big| \nonumber\\
	%
	%
	&=& \frac1{ \E Z_1}\sqrt{\frac{\phi_T(h)}{n}} \Big |
	\int_{T_{i-1}}^{T_i} \E\left(p(X_t)\Delta_t(x) \left[m(X_t,y)-m(x,y)\right] dt \Big | {\cal F}_{i-2}\right)  \Big |  \nonumber\\
	&\leq & \sqrt{n\phi_T(h)}\frac{p(x)|}{ n\E Z_1}
	\sup_{u\in {\cal B}(x, h)}|m(u,y)-m(x,y)| 
	\sup_{u\in {\cal B}(x, h)}|p(u)-p(x)|
	\Big | 
	\int_{T_{i-1}}^{T_i} \E\left(\Delta_t(x)  dt \Big | {\cal F}_{i-2}\right)  \Big |  \nonumber\\
	&=& O\left(\sqrt{n\phi_T(h)} \ h^\beta\right)O\left(\phi(h_T)\int_{T_{i-1}}^{T_i} f_{t, T_{i-2}}(x)dt +o(1)\right)\frac1{ n\E Z_1} \nonumber\\
		&=& O\left(\sqrt{n\phi_T(h)} \ h^\beta\right)O\left(\frac1T \int_{T_{i-1}}^{T_i} b_{t, \alpha_0}(x)dt)\right). \nonumber\\
\end{eqnarray*}
It follows by (A2)-(iii)  and the Cauchez-Schwarz inequality  that 
%

$$ \sum_{i=1}^n \left(\E \left[\eta_{T,i}(x,y) \ \Big | {\cal F}_{i-2} \right]\right)^2=O(h^{2\beta}\phi(h))=o(1).$$
Thus we have only to show that: 
$\sum_{i=1}^n \mathbb{E}[\eta^2_{T,i}( x,y) | \mathcal{F}_{
	i-2}] \overset{\mathbb{P}}{\rightarrow} \sigma^2( x,y)$.
Using again the Cauchez-Schwarz inequality, one may write
\begin{eqnarray}\label{Eval.eta2}
\sum_{i=1}^n \mathbb{E}[\eta^2_{T,i}( x,y) | \mathcal{F}_{i-2}]  &=& \frac1{ (\E Z_1)^2}\frac{\phi_T(h)}{n}%
\sum_{i=1}^n\E\left[  \left( \int_{T_{i-1}}^{T_i}   \zeta_t \Delta_t(x)[\psi_y(Y_t)-m(x,y)]       dt   \right)^2 \Big|  {\cal F}_{i-2}          \right] \nonumber\\
&\leq& \frac{\delta}{ (\E Z_1)^2}\frac{\phi_T(h)}{n}%
\sum_{i=1}^n\E\left[  \int_{T_{i-1}}^{T_i}   \zeta_t^2 \Delta^2_t(x)[\psi_y(Y_t)-m(x,y)]^2       dt   \Big|  {\cal F}_{i-2}          \right] \nonumber\\
&=& 
\frac{\delta}{ (\E Z_1)^2}\frac{\phi_T(h)}{n}%
\sum_{i=1}^n\E\left[  \int_{T_{i-1}}^{T_i}   \zeta_t^2 \Delta^2_t(x)[\psi_y(Y_t)-m(X_t,y)]^2       dt   \Big|  {\cal F}_{i-2}          \right] \nonumber\\
&+&
\frac{\delta}{ (\E Z_1)^2}\frac{\phi_T(h)}{n}%
\sum_{i=1}^n\E\left[  \int_{T_{i-1}}^{T_i}   \zeta_t \Delta^2_t(x)[m(X_t,y)-m(x,y)]^2       dt   \Big|  {\cal F}_{i-2}          \right] \nonumber\\
&=:&A_n+C_n
\end{eqnarray}

Conditioning three  times   with respects to ${\cal F}_{t-\delta}$ and ${\cal S}^\ell_{t-\delta}$, and using Conditions (A3)(ii-(iv)-(i$v^\prime$) and  the fact that $T=n\delta$, to get from   Lemma \ref{lem1}  that 
\begin{eqnarray}\label{An}
A_n&=& 
\frac{\delta}{ (\E Z_1)^2}\frac{\phi_T(h)}{n}%
\sum_{i=1}^n\E\left[  \int_{T_{i-1}}^{T_i}   p(X_t) \Delta^2_t(x) \overline{W}_2(X_t,y)       dt   \Big|  {\cal F}_{i-2}          \right] \nonumber\\
&\leq & 
\frac{\delta}{ (\E Z_1)^2}\frac{\phi_T(h)}{n}%
(p(x)+o(1))| (\overline{W}_2(x,y)+o(1))
\sum_{i=1}^n
\E\left[  \int_{T_{i-1}}^{T_i}    \Delta^2_t(x)       dt   \Big|  {\cal F}_{i-2}          \right] \nonumber\\
&\leq & 
(\delta+o(1)) p(x)\overline{W}_2(x,y)\frac{\phi^2_T(h)}{(\E Z_1)^2}%
\ \frac1n \sum_{i=1}^n
\int_{T_{i-1}}^{T_i}  \E\left[\frac1{\phi_T(h)}\Delta^2_t(x)          \Big|  {\cal F}_{i-2}          \right]dt \nonumber\\
&\leq & 
\delta(\delta+o(1)) p(x)\overline{W}_2(x)\frac{\phi^2_T(h)}{(\E(Z_1)^2}M_2
\left\{\frac1T \sum_{i=1}^n\int_{T_{i-1}}^{T_i}
f_{t, T_{i-2}}(x) dt+ 
O_{a.s.}\left[ \frac1T  \sum_{i=1}^n  \int_{T_{i-1}}^{T_i}   
\frac{ g_{t, T_{i-2},x}(h_T)}{\phi_T(h)} dt
\right] \right\}\nonumber\\ 
%
\end{eqnarray}

\vskip -6mm 
Using the Riemann's sum combined with condition (A2)(iii), it follows that 
$$\frac1T \sum_{i=1}^n\int_{T_{i-1}}^{T_i}
f_{t, T_{i-2}}(x) dt \leq \frac1T\int_0^T f_{t, T_{t-\delta}}(x)dt  \longrightarrow f(x) \quad a.s. \quad as \  T\longrightarrow \infty.
$$
and by (A2)(ii), which states that 
$\frac{ g_{t, T_{i-2},x}(h_T)}{\phi_T(h)}=o(1) \quad \text{as} \quad  T\longrightarrow \infty.$
Thus, 
\begin{eqnarray}\label{An.1}
A_n &\leq&  \delta(\delta+o(1)) p(x)\overline{W}_2(x,y)\frac{\phi^2_T(h)}{(\delta\phi_T(h_T) M_1 f(x)+o(1))^2}M_2 
\left[  f(x)+o(1) \right] \nonumber\\
&=& \frac{M_2}{M_1^2 f(x)}p(x)\overline{W}_2(x,y):=\tilde{\sigma}^2(x,y) \quad 
\text{as} \quad T\longrightarrow \infty.
\end{eqnarray}
Making use of the same arguments as above combined wi the fact that 
$\sup_{u\in {\cal B}(x, h)}|m(x)-m(u)|=h^{2\beta}$, we get $C_n=o_{a.s.}(1).$
%

\vskip 1mm

 \noindent {\bf Proof of part (b)}. Using successively H\"older, Markov,  Jensen and Minkowski inequalities combined with conditions  (A3)(iii), (A3)(i$v^\prime$) and Lemma  (\ref{lem1}), we get 
 , for any $\epsilon > 0$, any $p$ and $q$ such
 that $1 / p + 1 / q = 1$, that 
%
%
\begin{eqnarray*}
n \mathbb{E}[\xi_{T,i}^2( x,y) \mathbf{I}_{\{|\xi_{T,i}(x,y)| > \epsilon \}}] 
&\leq& 4n (\epsilon/2)^{-2q/p}\mathbb{E}[|\eta_{T,i}|^{2q}]=O\left(T\phi(h_T)^{-\gamma/2}\right)=o_{a.s}(1)
\end{eqnarray*}
by taking $2q=2+\gamma$ ($0<\gamma<1$), 
since $T\phi(h_T)$ towards to infinity as $T$ goes to infinity.

\end{proof}
  
 \bigskip  
\noindent  {\bf Proof of Corollary} \ref{colorlarryNor}.
We have 
\begin{eqnarray}\label{corllary-intr}
\sqrt{\frac{T F_{x,T}(h_T)}{\widetilde{V}^2_T(x,y)}}\left(\widehat{m}_T(x,y)-m(x,y) \right)= \sqrt{\frac{F_{x,T}(h_T)}{\phi(h_T) f(x)}} \sqrt{\frac{\sigma^2(x,y)f(x)}{\widetilde{V}^2_T(x,y)}} 
\sqrt{\frac{T\phi(h_T)}{ \sigma^2(x,y) }}\left(\widehat{m}_T(x,y)-m(x,y) \right)
\end{eqnarray}
We have from the consistency of  $F_{x,T}(h_T)$ and A2(i) that
$\frac{F_{x,T}(h_T)}{\phi(h_T) f(x)}$ goes to $1$ a.s.  as $T$ goes $+\infty$. By 
 Theorem \ref{AsymNormality} that quantity 
$\sqrt{\frac{T\phi(h_T)}{\sigma(x,y)}}
\left(\widehat{m}_T(x,y)-m(x,y) \right)
$ converges to ${\cal N}(0, 1)$ as $T\rightarrow \infty$. 
Using then the non-decreasing property of the cumulative  standard Gaussian distribution function $\Psi$, we get, 
for a given risk $0<\alpha<1$, the $(1-\alpha)-$ pseudo-confidence bands: 
\begin{eqnarray}\label{pseudo-intr0}
\sqrt{\frac{T\phi(h_T)}{\sigma^2(x,y)}}
\left|\widehat{m}_T(x,y)-m(x,y) \right| \leq \Psi^{-1}(1-\frac\alpha{2}).
\end{eqnarray} 
Considering now the statement (\ref{sigma1}) combined with 
Proposition \ref{PropoV(x)}, it  holds that 
\begin{eqnarray}\label{pseudo-intr}
\lim_{T\rightarrow \infty}\frac{\sigma^2(x,y)f(x)}{\widetilde{V}^2_n(x,y)}
\leq \lim_{T\rightarrow \infty} \frac{f(x)V^2(x,y)}{\widetilde{V}^2_n(x,y)}=
 \lim_{T\rightarrow \infty} \frac{\widetilde{V}^2(x,y)}{\widetilde{V}^2_n(x,y)}=
1 \quad a.s.,
\end{eqnarray}
because  $\widetilde{V}_n^2(x,y)$ is a consistent estimator of $\widetilde{V}^2(x,y)$. The proofs follows then from the statements
(\ref{corllary-intr}), (\ref{pseudo-intr0}) and (\ref{pseudo-intr})
 $\hfill \Box$



\bigskip 
\noindent {\bf  Proof of Theorem \ref{uniform} }

Letting 
$\lambda_T=\sqrt{\frac{\varphi_{\cal C}(\epsilon_T)}{T\phi(h_T)} }$
with $\epsilon_T= \frac{\log T}{T}$ and consider 
  $n\delta=T$, $n=[T] \  (T\geq 1)$, consequently $1\leq \delta<2$. Observe that 
\begin{eqnarray}\label{uniform.eq1}
&\sup_{y\in S} \sup_{x\in \C}|\widehat{m}_T(x,y) - m(x,y)| \leq \nonumber\\ 
& \sup_y\sup_{x\in \C}|B_T(x,y)| +
\frac{\sup_{y\in S}\sup_{x\in \C}|R_T(x,y)| + \sup_{y\in S}\sup_{x\in \C}|Q_T(x,y)|}{\inf_{x\in \C}|\widehat{m}_{T,1}(x)|}.
\end{eqnarray}
Since
$\inf_{x\in \C}\big| \widehat{m}_{T,1}(x) \big| > \inf_{x\in \C}\big| p(x) \big| - \sup_{x\in \C}\big| \widehat{m}_{T,1}(x) - p(x) \big|,$
using the same steps of the proof of Lemma \ref{mhat}, one can show,
under conditions (A1), (U0), (A3)(iv)-(i$v^\prime$), (A3)(i)-(ii), the second term  in the above inequality  equals zero. Thus, making use (U0)(v), we get for sufficiently large $T$ that      $\inf_{x\in \C}\Big| \widehat{m}_{T,1}(x) \Big|>b_1$ a.s. 

Concerning  the conditional bias, inspection of  the proof of the statement (\ref{biastilde}) 
shows that the term  $\tilde{B}_T(x,y)$ is bounded above by a constant which  is independent of $x$ and $y$, therefore,  we have under  (A3)(i)-(ii) that
\begin{eqnarray}\label{BiasCond} 
\sup_y	\sup_{x\in \C}\big|B_T(x,y)\big|  
%
\leq \frac{\sup_y \sup_{x\in \C}\big|\widetilde{B}_{T}(x,y)\big|}{\inf_{x\in \C}\big|\overline{m}_{T,1}(x)\big|} \leq \frac1{b_1}\sup_y \sup_{x\in \C}\big|\widetilde{B}_{T}(x,y)\big|=\OH_{a.s.}(h_T^\beta).
\end{eqnarray}
Moreover,  under the assumption of Lemma \ref{U1}, making use of 
the decomposition (\ref{Q}), it follows from  
(U0)(iv) that $\sup_{y\in S}\sup_{x\in \C}|Q_T(x,y)| = \OH_{a.s.}(\lambda_T)$. In the other hand   conditions (\ref{R}) and 
(\ref{BiasCond}) allow to conclude   $\sup_{y\in S}\sup_{x\in \C}|R_T(x,y)| = \OH_{a.s.}\left(\lambda_T h_T^\beta\right)$.
The mean task is  to prove the following Lemma that allows with the statement  (\ref{decomp}) to achieve the proof of the Theorem \ref{uniform}  $\hfill$ 

\begin{lemma}\label{U1} Assume that  (A1), (U0)(i)-(iii), (A3)(i), (A3)(iii-iv), (A3)(i$v^\prime$), (U1)-(U4)  together with conditions  (\ref{cond})  and    (\ref{CEntropy}) are satisfied. Then  we have
	$$\sup_{y\in S}\sup_{x\in \C}|\widehat{m}_{T,2}(x,y) - \overline{m}_{T,2}(x,y)| = \OH_{a.s.}(\lambda_T).$$
\end{lemma}

\noindent	{\bf Proof of Lemma \ref{U1}  }.

Let  $\epsilon>0$ be given and  	consider a covering of the class  of functions ${\cal C}$ by closed  balls  

$B(c_k, \epsilon) = \{x\in \C: d(x, c_k) < \epsilon \},  \ 1\leq k\leq =\N(\epsilon, \C, d):=N_\epsilon,$
that is 	${\cal C} \subset  \cup_{k=1}^{N_\epsilon} B(c_k, \epsilon)$. Then we have 
\begin{eqnarray}\label{DD}
\sup_{y\in S}\sup_{x\in \C}|\widehat{m}_{T,2}(x,y) - \overline{m}_{T,2}(x,y)| &\leq &
%
%
\sup_{y\in S}	\max_{1\leq k \leq N_\epsilon }\sup_{x\in B(c_k, \epsilon)}|\widehat{m}_{T,2}(x,y) - \widehat{m}_{T,2}(c_k, y)| \nonumber \\
& +&	\sup_{y\in S} \max_{1\leq k \leq N_\epsilon}|\widehat{m}_{T,2}(c_k,y) - \overline{m}_{T,2}(c_k,y)| \nonumber\\
&& + 	\sup_{y\in S} \max_{1\leq k \leq N_\epsilon}\sup_{x\in B(c_k, \epsilon)}|\overline{m}_{T,2}(x,y) - \overline{m}_{T,2}(c_k,y)| \nonumber\\
&=:& \mathcal{H}_{T,1} + \mathcal{H}_{T,2} + \mathcal{H}_{T,3}.
\end{eqnarray}
Let us now focus on  the  first term $\mathcal{H}_{T,1}$.  We have  for any $x \in B(c_k, \epsilon)$ and $y\in S$ 
\begin{eqnarray}\label{HT1}
\widehat{m}_{T,2}(x,y) - \widehat{m}_{T,2}(c_k,y) &=& \frac{1}{n\E(Z_1(x))} \int_0^T \zeta_t \psi_y(Y_t) \left[\Delta_t(x) - \Delta_t(c_k)\right] dt  \nonumber \\
&+&\frac{1}{n\E(Z_1(x)) \E(Z_1(c_k))} \int_0^T \zeta_t \psi_y(Y_t) \Delta_t(c_k) \left[ \E(Z_1(c_k)) - \E(Z_1(x)) \right] dt \nonumber\\
&:=& \mathcal{I}_{T,1}(c_k,y) + \mathcal{I}_{T,2}(c_k,y).
\end{eqnarray}

\bigskip 
Making use of   the property of ergodicity, conditions $(U_1)$  and  $(U_2)$ and  the boundedness of $\zeta$, we get for $T$ sufficiently large and   any $(x, y)\in B(c_k, \epsilon)\times S$ that
\begin{eqnarray}\label{I1}
|\mathcal{I}_{T,1}(c_k,y)| 
&\leq& 
a_3 c_\zeta \frac{\epsilon}{h_T}\frac1{\E(\Delta_0(x)) } \frac1T \int_0^{T}  |\psi_y(Y_t)|dt  \nonumber\\
%
%
%
&\leq& \frac{a_3c_\zeta}{a_2}\frac{\epsilon}{h_T}
\OH_{a.s.} \sup_{y\in S} \E\left( |\psi_y(Y_0)|\right)
=\OH_{a.s.} \left(\frac{\epsilon}{h_T}\right).
\end{eqnarray}	

On the other hand, we have under the above conditions 
\begin{eqnarray}\label{I2}
&&|\mathcal{I}_{T,2}(c_k,y)| \leq  	
c_\zeta\frac{\E\left(|\Delta_0(c_k) - \Delta_0(x)|\right) }{\E(\Delta_0(x)) } \frac1{n\E(Z_1(c_k))}
\int_0^T | \psi_y(Y_t)| \Delta_t(c_k)  dt
\nonumber	\\
&\leq& \frac{c_\zeta a^2_3}{a^2_2} \frac{\epsilon}{h_T}   
\frac1T
\int_0^T | \psi_y(Y_t)|  dt=
\frac{c_\zeta a^2_3}{a^2_2} \frac{\epsilon}{h_T}\sup_{y\in S} \E\left( |\psi_y(Y_0)|\right)  
=\OH_{a.s.} \left(\frac{\epsilon}{h_T)}\right).
\end{eqnarray}	
\noindent 	The constants in the right hand side of (\ref{I1}) and (\ref{I2})
are independent of 	$x$ and $y$, thus 
\begin{eqnarray}\label{I3some}
{\cal H}_{T,1}=\sup_{y\in S}	\max_{1\leq k \leq N_\epsilon }\sup_{x\in B(c_k, \epsilon)}	|\mathcal{I}_{T,1}(c_k,y) + \mathcal{I}_{T,2}(c_k,y)|=
O_{a.s.}\left(\frac{\epsilon_T}{h_T }\right)
\end{eqnarray}
Similarly we get, under the above conditions,    the same  bound of the 
term $\mathcal{H}_{T,3}$,  . Using  condition (U4) we conclude, for $T$ sufficiently  large enough,    that 	
\begin{eqnarray}\label{I3max}
\mathcal{H}_{T,1}=  	\mathcal{H}_{T,3}=O_{a.s.}\left(\lambda_T \right)  \quad \text{with} \quad  \epsilon_T=\frac{ \log T}{T}
\end{eqnarray}
Consider now, 	 the intermediate  term $\mathcal{H}_{T,2}$. For this purpose,   cover $S$ with $\nu_T= [T^{\gamma}]+1$
intervals $I_{k\prime}=[\ell_n-y_{k^\prime}, \ \ell_n +y_{k^\prime}]$ of centre $y_{k^\prime} \in S$
and length $2\ell_T \leq c/\nu_T$, for some $\gamma>0$, such that  
$S\subset \cup_{k'=1}^{\nu_T}I_{k^\prime}$.  Then we have
\begin{eqnarray}
{\cal H}_{T,2}&=& \sup_{y\in S}
\max_{1\leq k \leq N_\epsilon}|\widehat{m}_{T,2}(c_k,y) - \overline{m}_{T,2}(c_k,y)| \nonumber\\
&\leq& 
\max_{1\leq k^\prime\leq \nu_n}\ \max_{1\leq k \leq N_\epsilon}\   \sup_{y_k \in I_{k^\prime}} \ 
|\widehat{m}_{T,2}(c_k,y) - \widehat{m}_{T,2}(c_k,y_{k^\prime})| \nonumber\\
&+& 
\max_{1\leq k^\prime\leq \nu_T}\ \max_{1\leq k \leq N_\epsilon}\   
|\widehat{m}_{T,2}(c_k,y_{k^\prime}) - \overline{m}_{T,2}(c_k,y_{k^\prime})| \nonumber\\
&+& 
\max_{1\leq k^\prime \leq \nu_T}\ \max_{1\leq k \leq N_\epsilon}\   \sup_{y_k \in I_{k^\prime}} \ 
|\overline{m}_{T,2}(c_k,y_{k^\prime}) - \overline{m}_{T,2}(c_k,y)| 
\nonumber\\
&:=& {\cal R}_{1,T}+{\cal R}_{2,T}+{\cal R}_{3,T}
\end{eqnarray}

Using  (U1)(ii), (U2)  and (U3) combined with the ergodic property  and because $\zeta$ is bounded, one may write 
\begin{eqnarray}\label{sup-y1}
|\widehat{m}_{T,2}(c_k,y) - \widehat{m}_{T,2}(c_k,y_{k^\prime})|
&=&[T\mathbb{E}(\Delta_0(c_k))]^{-1}\int_0^T |\zeta_t| \Delta_t(c_k)\left| \psi_y(Y_t)-
\psi_{y_{k^\prime}}(Y_t)\right|dt \nonumber\\
&\leq& c_\psi \frac{a_3}{a_2}\frac{1}{\nu_T^{\gamma}}\mathbb{E}(|\zeta_0|)
=\OH_{a.s.}(T^{-\gamma}).
%
%
\end{eqnarray}
The same bound may be obtained for the quantity 	$\overline{m}_{T,2}(c_k,y_{k^\prime}) - \overline{m}_{T,2}(c_k,y)$. 
Since the constants in the above terms are  independent of $c_k$ and $y_{k^\prime}$, we conclude  that
\begin{eqnarray}\label{R1-R3max}
{\cal R}_{1,T}={\cal R}_{3,T}=
O_{a.s}(\lambda_T)
\end{eqnarray}
because $\lim_{T \to +\infty}\frac1{\lambda_T \nu_T^{\gamma}}=0$ in view of condition (U4). 

We turn now to the  intermediate term ${\cal R}_{2,T}$. Using the fact that $\varphi_{\cal C}(\epsilon_n)=\log N_\epsilon$, $\nu_T=[T^\gamma]+1$, 
the statement  (\ref{Pointwise}) which still  true under conditions  
(A1), (U0)(i)-(iii), (A3)(i), (A3)(iii-iv) and (A3)(i$v^\prime$), we get for any  $\epsilon_0>0$, whenever  condition (\ref{cond}) and the assumption  (U4) are satisfied,  that 
\begin{eqnarray}\label{unifom.eq3}
\mathbb{P}({\cal R}_{2,T}>\lambda_T)=
& & \mathbb{P}\left( 	\max_{1\leq k^\prime\leq \nu_T}\ \max_{1\leq k \leq N_\epsilon}\   
|\widehat{m}_{T,2}(c_k,y_{k^\prime}) - \overline{m}_{T,2}(c_k,y_{k^\prime})|>\lambda_T\right)\nonumber\\
&\leq & \sum_{k^\prime=1}^{\nu_T}\sum_{k=1}^{N_\epsilon} \mathbb{P}\left( |\widehat{m}_{T,2}(c_k) - \overline{m}_{T,2}(c_k)| > \lambda_T \right) \nonumber\\
&\leq& 2 \nu_T N_{\epsilon_n} \exp\left\{ -c\epsilon_0^2\varphi_C(\epsilon_n)  \right\}
= 2 \nu_T N_{\epsilon_n}^{1-c\epsilon_0^2} \nonumber\\
&\leq & 2n^{\gamma} N_{\epsilon_n}^{1-c\epsilon_0^2}
%
%
%
\end{eqnarray}
because $T=n\delta$ and   $1\leq \delta<2$. 
Choosing $c\epsilon_0^2=\eta$ and considering condition (\ref{CEntropy}) which is equivalent to 
$$\sum_{n\geq 1} n^{\gamma} N_{\epsilon_n}^{ 1-\eta}<\infty \quad \text{for some} \  \eta>0$$
to conclude by Borel-Cantelli  Lemma the end of the proof of Lemma 
\ref{U1} $\hfil \Box$. 


\bigskip

{}


\end{document}